\documentclass[porbanz,preprint]{imsart-porbanz}
\pdfoutput=1

\RequirePackage[OT1]{fontenc}
\RequirePackage{amsthm,amsmath}
\RequirePackage[colorlinks,linkcolor=blue,citecolor=blue]{hyperref}

\usepackage{amsthm,amsmath}
\usepackage[nohug]{diagrams}
\usepackage{picins}

\usepackage[numbers]{natbib}

\startlocaldefs

\numberwithin{equation}{section}
\theoremstyle{plain}
\pdfoutput=1

\theoremstyle{plain}

\newtheorem{theorem}{Theorem}
\newtheorem{lemma}{Lemma}
\newtheorem{proposition}{Proposition}
\newtheorem{corollary}{Corollary}
\theoremstyle{definition}
\newtheorem{definition}{Definition}
\theoremstyle{remark}

\newtheorem{example}{Example}

\def\condind{\perp\!\!\!\perp}

\def\ie{i.e.\ }
\def\eg{e.g.\ }

\def\mean#1{\mathbb{E}[#1]}

\def\eqae{=_{\mbox{\tiny a.e.}}}

\def\equae{=_{\mbox{\tiny a.e.}}}

\def\inclusion{\mathcal{J}}
\def\inclusionX{\inclusion_{\xspace}}
\def\wstar{weak$^{\ast}$ }
\def\symmdiff{\!\vartriangle\!}

\def\indI{\mbox{\tiny I}}
\def\indJ{\mbox{\tiny J}}
\def\indK{\mbox{\tiny K}}
\def\indJI{\mbox{\tiny J$\setminus$I}}
\def\indE{\mbox{\tiny E}}

\def\indD{\mbox{\tiny D}}
\def\indi{\mbox{\tiny{\{i\}}}}
\def\ind#1{\mbox{\tiny #1}}

\def\parts{\mathcal{H}}
\def\partsQ{\parts(\mathcal{Q})}

\def\abstspace{\Omega}
\def\xspace{\mathcal{X}}
\def\yspace{\mathcal{Y}}
\def\tspace{\mathcal{T}}
\def\xspaceI{\xspace_{\indI}}
\def\xspaceJ{\xspace_{\indJ}}
\def\xspaceD{\xspace_{\indD}}

\def\tspaceI{\tspace_{\indI}}
\def\tspaceJ{\tspace_{\indJ}}
\def\tspaceD{\tspace_{\indD}}

\def\txspace{\tilde{\xspace}}
\def\yspaceI{\yspace_{\indI}}
\def\yspaceJ{\yspace_{\indJ}}
\def\yspaceD{\yspace_{\indD}}
\def\yspaceE{\yspace_{\indE}}
\def\txspace{\tilde{\xspace}}
\def\ttspace{\tilde{\tspace}}
\def\xI{x_{\indI}}
\def\xJ{x_{\indJ}}
\def\xD{x_{\indD}}

\def\tImage{\Gamma}
\def\simp{\triangle}
\def\simpI{\simp_{\indI}}
\def\simpJ{\simp_{\indJ}}

\def\AI{A_{\indI}}

\def\pMeas{M}

\def\cfspace{C}

\def\borelV{\borel_{V}}

\def\borel{\mathcal{B}}
\def\top{\mbox{Top}}

\def\borelI{\borel_{\indI}}

\def\borelD{\borel_{\indD}}

\def\tborel{\tilde{\borel}}
\def\abstfield{\mathcal{A}}
\def\field{\mathcal{C}}
\def\fieldI{\field_{\indI}}
\def\fieldJ{\field_{\indJ}}

\def\fieldD{\field_{\indD}}

\def\Sfield{\mathcal{S}}
\def\SfieldI{\mathcal{S}_{\indI}}

\def\SfieldD{\mathcal{S}_{\indD}}

\def\borelx{\borel_x}

\def\borely{\borel_{y}}
\def\borelT{\borel_{\tspace}}
\def\borelS{\borel_s}
\def\topI{\top_{\indI}}
\def\topJ{\top_{\indJ}}
\def\topD{\top_{\indD}}

\def\XI{X_{\indI}}
\def\Xi{X_{\ind{i}}}

\def\ThetaI{\Theta_{\indI}}
\def\XJ{X_{\indJ}}
\def\ThetaJ{\Theta_{\indJ}}
\def\XD{X_{\indD}}
\def\ThetaD{\Theta_{\indD}}

\def\tX{\tilde{X}}
\def\tTheta{\tilde{\Theta}}

\def\SI{S_{\indI}}

\def\rest{\phi}

\def\projector{\mbox{pr}}

\def\projectorI{\projector_{\indI}}
\def\projectorJI{\projector_{\indJ\indI}}

\def\po{\preceq}
\def\famD#1{{\bigl< #1 \bigr>}_{\indD}}

\def\fJI{f_{\indJ\indI}}
\def\fKI{f_{\indK\indI}}
\def\fKJ{f_{\indK\indJ}}
\def\fII{f_{\indI\indI}}
\def\fI{f_{\indI}}
\def\fJ{f_{\indJ}}

\def\gJI{g_{\indJ\indI}}
\def\gI{g_{\indI}}

\def\hJI{h_{\indJ\indI}}
\def\hI{h_{\indI}}
\def\hJ{h_{\indJ}}

\def\plim{\varprojlim}

\def\abstmeasure{\mathbb{P}}
\def\P{P}
\def\PI{P_{\indI}}
\def\PJ{P_{\indJ}}
\def\PD{P_{\indD}}

\def\PX{P_{\mbox{X}}}

\def\PXI{P_{\XI}}

\def\tP{\tilde{P}}

\def\SI{S_{\indI}}
\def\SJ{S_{\indJ}}

\def\tk{\tilde{k}}
\def\kI{k_{\indI}}

\def\DP#1{\mbox{DP}\left( #1 \right)}

\def\tyspace{\tilde{\yspace}}
\def\tF{\tilde{F}}

\def\tmodel{\tilde{\model}}
\def\tnu{\tilde{\nu}}

\def\PTheta{P^{\theta}}

\def\borelY{\borel_{\yspace}}

\def\PX{P^{x}}
\def\PXI{\PX_{\indI}}

\def\PThetaI{\PTheta_{\indI}}
\def\PThetaJ{\PTheta_{\indJ}}
\def\PThetaD{\PTheta_{\indD}}
\def\YI{Y_{\indI}}
\def\YJ{Y_{\indJ}}
\def\YD{Y_{\indD}}
\def\Tn{T^{(n)}}
\def\indexspace{\mathcal{W}}
\def\tyspace{\tilde{\yspace}}
\def\tY{\tilde{Y}}
\def\inclusionT{\inclusion_{\tspace}}
\def\tPTheta{\tilde{P}^{\theta}}
\def\tTn{\tilde{T}^{(n)}}
\def\inclusionY{\inclusion_{\yspace}}

\def\tyspace{\tilde{\yspace}}
\def\tF{\tilde{F}}

\def\tmodel{\tilde{\model}}
\def\tnu{\tilde{\nu}}
\def\tOmega{\tilde{\abstspace}}
\def\tabstmeasure{\tilde{\abstmeasure}}
\def\model{\mathcal{P}}

\def\tf{\tilde{f}}
\def\tx{\tilde{x}}

\def\ty{\tilde{y}}

\def\Sspace{\mathcal{U}}
\def\Sfield{\mathcal{S}}
\def\SfieldI{\mathcal{S}_{\indI}}

\def\SfieldD{\mathcal{S}_{\indD}}
\def\borelS{\borel(\Sspace)}
\def\borelSI{\borel(\SspaceI)}
\def\SspaceI{\Sspace_{\indI}}

\def\tS{\tilde{S}}
\def\tA{\tilde{A}}
\def\tk{\tilde{k}}
\def\tC{\tilde{C}}
\def\hJI{h_{\indJ\indI}}
\def\SD{S_{\indD}}
\def\kI{k_{\indI}}
\def\kD{k_{\indD}}

\def\kernel{k}
\def\tauI{\tau_{\indI}}
\def\tauJ{\tau_{\indJ}}
\def\tauD{\tau_{\indD}}
\def\thetaI{\theta_{\indI}}
\def\thetaJ{\theta_{\indJ}}
\def\thetaD{\theta_{\indD}}
\def\kI{k_{\indI}}
\def\kJ{k_{\indJ}}
\def\kD{k_{\indD}}
\def\modelI{\model_{\indI}}
\def\SG#1{\mathbb{S}_{#1}}
\def\SGinf{\mathbb{S}_{\infty}}
\def\VP{\mathfrak{S}}
\def\fnn{f_{n+1,n}}
\def\gammaI{\gamma_{\indI}}
\def\gammaD{\gamma_{\indD}}
\def\chull{\mbox{conv}}
\def\TnI{\Tn_{\indI}}
\def\TnD{\Tn_{\indD}}
\def\kernelI{\kernel_{\indI}}
\def\kernelD{\kernel_{\indD}}
\def\yD{y_{\indD}}
\def\yI{y_{\indI}}

\def\phiI{\phi_{\indI}}

\usepackage[normalem]{ulem}
\usepackage{verbatim}

\endlocaldefs

\usepackage{amsmath,amssymb,amscd,amsfonts,amsthm}

\begin{document}

\begin{frontmatter}

\title{Conjugate Projective Limits}
\runtitle{Conjugate Projective Limits}
\author{\fnms{Peter} \snm{Orbanz}\corref{}\ead[label=e1]{p.orbanz@eng.cam.ac.uk}}
\affiliation{University of Cambridge}
\runauthor{P.\ Orbanz}


\begin{abstract}
  \pdfoutput=1

We characterize conjugate nonparametric Bayesian models as
projective limits of conjugate, finite-dimensional Bayesian
models. In particular, we identify a large class of nonparametric
models representable as infinite-dimensional analogues
of exponential family distributions and their canonical
conjugate priors. This class contains most models studied
in the literature, including Dirichlet processes and 
Gaussian process regression models. 
To derive these results, we
introduce a representation of infinite-dimensional Bayesian
models by projective limits of regular conditional probabilities.
We show under which conditions the nonparametric model itself,
its sufficient statistics, and -- if they exist -- conjugate
updates of the posterior are projective limits of their respective
finite-dimensional counterparts.
The results are illustrated both by application to existing nonparametric
models and by construction of a model on infinite permutations.

\end{abstract}

\end{frontmatter}

\pdfoutput=1

\section{Introduction}

Nonparametric Bayesian statistics effectively
revolves around a small number of fundamental models, including
the Dirichlet process \cite{Ferguson:1973}, Gaussian process \cite{Zhao:2000,Rasmussen:Williams:2006},
beta process \cite{Hjort:1990} and gamma process \cite{Ferguson:1973}. 
All these models have conjugate posteriors
\cite{Walker:Damien:Laud:Smith:1999}.
Since most nonparametric Bayesian models are derived from such fundamental, conjugate models,
virtually all nonparametric Bayesian inference is based directly or indirectly on conjugacy.
The objective of this work is to study the shared properties of fundamental models and to 
characterize the class of models admitting conjugate posteriors.

By \emph{nonparametric Bayesian model}, we refer to a Bayesian
model on an infinite-dimensional parameter space
\citep{Hjort:et:al:2010,Ghosh:Ramamoorthi:2002,Walker:Damien:Laud:Smith:1999}.
We do not a priori distinguish between
discrete models (\eg Dirichlet processes)
and continuous models (\eg Gaussian process regression).
In addition to conjugacy, 
models such as the Gaussian and Dirichlet processes share
another property, the existence of marginals in the exponential family.
In the case of the Dirichlet process, there is a well-known
connection between the two properties:
Conjugacy of the nonparametric model can be derived
directly from the conjugacy of the marginal, finite-dimensional
Dirichlet priors \citep{Ghosal:2010}.
We will show in the following how the vague but intuitively appealing
link between conjugate posteriors and exponential family marginals
in general nonparametric Bayesian models can be made precise.
If an infinite-dimensional model is constructed from finite-dimensional
marginal distributions, conjugacy of the marginals proves sufficient
to guarantee a conjugate posterior of the nonparametric model.

The analysis of shared properties of models requires a shared representation,
which leads almost inevitably to projective limits,
\ie the representation of a stochastic process by its
finite-dimensional marginal distributions \cite{Bourbaki:2004}.
Most representations used in Bayesian nonparametrics are
adapted to specific models -- examples include
L\'{e}vy processes, stick-breaking constructions \citep{Sethuraman:1994},
transformed Poisson processes \citep{Ferguson:Klass:1972},
and normalized completely random measures \citep{James:Lijoi:Pruenster:2009}.
The advantages of such model-specific representations
are that they emphasize useful properties of the model
in question, as well as their simplicity -- more general
representations tend to come at the price
of more technical subtleties involved in their application.
Possible choices for more general representations of probability
measures are densities, characteristic functions and projective
limits. Densities are not applicable for nonparametric
Bayesian models, both for lack of a suitable translation-invariant
carrier measure on infinite-dimensional space, and because
some important models (such as the Dirichlet process) 
are not dominated \citep{Schervish:1995}.
Characteristic functions 
are ill-suited for the questions considered here,
since they do not live on the actual sample space. 

A \emph{projective limit} (also called an \emph{inverse limit})
assembles an infinite-dimensional mathematical object
from a family of finite-dimensional objects 
\citep{Bourbaki:1968,Bourbaki:1966,Bourbaki:2004}. 
Projective limits of probability measures,
\ie Kolmogorov's extension theorem and its generalizations, are
widely used in the construction of stochastic processes:
A stochastic process with paths in an infinite-dimensional space is
represented in terms of its finite-dimensional marginals
\citep{Kallenberg:2001}. Since a projective limit representation
is not sufficient to specify some important properties
of sample paths, such as continuity of random functions
or $\sigma$-additivity of random measures, we combine
projective limits with the notion of a
\emph{pullback} under a suitable
transformation mapping \citep{Fremlin:MT}.
The pullback accounts for those almost sure properties of 
paths not expressible in terms of the projective limit.

Projective limits can be defined not only for measures, but
also for sets, functions, and a wide variety of mathematical
structures
\citep{Bourbaki:1968,Bourbaki:1966,Bourbaki:2004,MacLane:1998}.
This allows us to both define projective limits of conditional
probabilities, and to apply the representation to sufficient
statistics and other functions associated with a model.
In this manner, we obtain a representation of a nonparametric
Bayesian model in terms of a 
family of finite-dimensional ``marginal'' Bayesian models.
The properties of the nonparametric model can be related directly
to those of the parametric marginals. Application to
the questions of sufficiency and conjugacy shows
that both the sufficient statistics
and the posterior updates of a nonparametric Bayesian model
can be expressed in terms of their finite-dimensional counterparts.
This result in particular establishes a large family
of models -- containing both the Gaussian and
the Dirichlet process -- which can be regarded as a nonparametric
analogue of the exponential family, in a sense to be made precise
in the ensuing discussion.

The results imply an approach to the construction from scratch 
of nonparametric Bayesian models on a wide range of domains. In this regard, an additional appeal
of projective limits is the large number of such representations available in the
mathematical literature, each of which may potentially be harvested for the purposes of Bayesian
nonparametrics. Examples include the projective
limit/pullback construction of continuous functions 
used in the construction of the Gaussian process \citep[e.g.][]{Bauer:1996};
a variety of constructions of topological and algebraic objects discussed
by \citet{Bourbaki:1966,Bourbaki:1968,Bourbaki:2004}; the construction
of random coagulation and fragmentation processes \cite{Bertoin:2006};
and recent constructions
of infinite limits of permutations by \citet{Kerov:Olshanski:Vershik:2004}, and
of graph limits by \citet{Lovasz:Szegedy:2006}.

\subsection{Summary of Results}

\newcounter{counter:results}
\setcounter{counter:results}{0}
\def\resultcounter{\renewcommand{\labelenumi}{(\arabic{counter:results})}
  \addtocounter{counter:results}{+1}}

Since projective limits are, by themselves, not capable of
expressing all properties of stochastic processes such
as the Dirichlet and Gaussian process, additional
steps are required to obtain an applicable distribution.
These steps and their formalization in the literature
differ widely between models.
Since our problem requires a unified formalism, we
derive a representation in terms of a pullback
of the projective limit under a measurable embedding.
Intuitively, the stochastic process of interest is represented
by uniquely encoding each of its paths as a path
of the projective limit process. The resulting 
representation is applicable to
all important nonparametric Bayesian models.

Projective limits and pullbacks preserve a variety of
properties of functions and set functions. For example, projective
limits and pullbacks obtained from injective functions are again
injective functions. The same holds for continuous and measurable mappings,
bijections, probability measures and regular conditional
probabilities. Some of these facts are standard results,
others are established in the following. In particular,
we show:
\begin{enumerate}
\resultcounter
\item The countable projective limit of a projective family
  of probability kernels (regular conditional probabilities) on
  finite-dimensional spaces is a probability kernel on an
  infinite-dimensional space. 
  The extension theorems of Kolmogorov and of Prokhorov 
  can both be generalized along these lines 
  (Theorem \ref{theorem:projlim:conditionals}; Corollary \ref{corollary:prokhorov:extension}).
  Similarly, the pullback of a probability kernel is again a probability kernel
  (Proposition \ref{lemma:pullback:conditionals}).
\end{enumerate}
A Bayesian model is defined by conditional probabilities. 
By application of the previous results to these 
conditionals, we obtain:
\begin{enumerate}
\resultcounter
\item A projective limit can be applied directly to finite-dimensional
  Bayesian models, resulting in infinite-dimensional Bayesian models on the
  corresponding projective limit spaces
  (Sec.~\ref{sec:bayesian:projlim}). Pullbacks also preserve the structure
  of the Bayesian model (Sec.~\ref{sec:bayesian:pullback}). 
  Both operations commute with the computation of posteriors 
  (Diagram \eqref{diagram:posterior:commutes}).
\end{enumerate}  
In other words, nonparametric Bayesian models can
be directly constructed from finite-dimensional ``marginal'' Bayesian models.
The construction is analogous to the construction of stochastic process measures by means
of projective limits and pullbacks.

Since projective limits and pullbacks are applicable to measurable functions,
they apply simultaneously to a model and its associated statistics.
\begin{enumerate}
\resultcounter
\item The projective limit of the sufficient statistics (resp. sufficient $\sigma$-algebras)
  of the marginal models is a sufficient statistic (resp. sufficient $\sigma$-algebra)
  of the infinite-dimensional projective limit model (Sec. \ref{sec:sufficiency}).
  We also show that, if the sufficient $\sigma$-algebras of the marginals are
  minimal, the projective limit $\sigma$-algebra is again minimal sufficient. This
  holds even if the projective limit model is undominated (Proposition
 \ref{theorem:suffstat:minimal}).
\end{enumerate}
The practical utility of conjugate Bayesian models is due to the representability of their
posterior parameters as functions of the data and the model hyperparameters. We show
that the structure and functional form of this update process carries over from
the marginals to the nonparametric model.
\begin{enumerate}
\resultcounter
\item Projective limits and pullbacks of conjugate Bayesian models are conjugate, and
  in particular, the mapping to the posterior parameter of the infinite-dimensional
  model is the projective limit of the update mappings of the marginal models
  (Sec.~\ref{sec:conjugacy}).
  For the specific case in which the finite-dimensional marginals are conjugate
  exponential family models, we obtain a nonparametric analogue of
  the Diaconis-Ylvisaker representation \citep{Diaconis:Ylvisaker:1979} of
  conjugate parametric models (Corollary \ref{corollary:expfam}).
\end{enumerate}
The results are illustrated by application to three concrete examples:
Gaussian processes (Examples \ref{example:GP:C} and 
\ref{example:GP:regression}), Dirichlet processes
(Example \ref{example:DP} and Sec.~\ref{sec:examples:DP}), 
and a Bayesian model
on infinite permutations (Sec.~\ref{sec:examples:cayley}).

\subsection{Related Work}

The application of projective limits to statistical models
was pioneered by \citet{Lauritzen:1984,Lauritzen:1988},
to derive a family of parametric models which are
defined by 
sequences (rather than averages) of sufficient statistics
and generalize beyond exchangeable observations.
In Lauritzen's work, the ``dimensions'' of the projective limit
describe repeated observations from a parametric model, rather
than dimensions of sample and parameter space as in our case.
Nonetheless, if $n$ observations in Lauritzen's ``projective
statistical fields'' \citep[][Chapter IV]{Lauritzen:1988} are
interpreted as a sample of size $n$ in a Bayesian nonparametric
model, the projective limit aspects of Sec.~\ref{sec:conditionals}
below can be regarded as an analogue of Lauritzen's projective
fields for application to nonparametric Bayesian models.

Conjugate analysis in the finite-dimensional, parametric case,
\ie for dominated models, is the subject of a substantial
literature \citep[e.g.][]{Diaconis:Ylvisaker:1979,Dalal:Hall:1983,
Consonni:Veronese:2003}.
\citet{Bernardo:Smith:1993} give a concise overview.
It is also well known that almost all nonparametric Bayesian
models are conjugate \cite{Walker:Damien:Laud:Smith:1999};
if the model is undominated, Bayes' theorem
is not applicable, and conjugacy is often the only
way to represent the posterior. Other models indirectly
rely on conjugacy: The popular Dirichlet
process mixture model \citep[Example 4]{Antoniak:1974}
does not have a conjugate posterior, but is amenable
to Gibbs sampling only because the Dirichlet process law of
the mixing measure is conjugate. However, conjugacy of
nonparametric Bayesian models has
not so far been analyzed as a structural property,
with one notable exception:
In the special case of sequential independent increment processes,
for which a class of models with exponential family marginals 
is discussed in detail
by \citet{Kuechler:Sorensen:1997}, the
existence of conjugate posteriors is studied by
\citet{Magiera:Wilczynski:1991}.
\citet{Thibeaux:Jordan:2007} draw on a similar insight
and invoke a conjugacy argument to relate
the Indian buffet process model of
\citet{Griffiths:Ghahramani:2006} to the 
beta process of \citet{Hjort:1990}.

\subsection{Outline}

We develop a representation of stochastic processes
suitable for our purposes in Sec.~\ref{sec:stochproc}. 
Projective limits and pullbacks are then applied to 
conditional probabilities in Sec.~\ref{sec:conditionals},
which facilitates their application to Bayesian models
in Sec.~\ref{sec:bayesian}. From the representation
of nonparametric Bayesian models so obtained, we derive
results on their sufficient statistics in 
Sec.~\ref{sec:sufficiency}, and on 
conjugate posteriors in Sec.~\ref{sec:conjugacy}.
Two detailed examples in Sec.~\ref{sec:examples}
illustrate the approach and results.
Since projective limits of functions and pullbacks
of measures are not commonly used in statistics, a brief
summary of relevant facts is provided in Appendix \ref{sec:background}.

\subsection{Notation and Assumptions}

All random variables are in the following assumed to share
an abstract probability
space $(\abstspace,\abstfield,\abstmeasure)$ as common domain.
We will frequently have to distinguish spaces of different
dimensions, which are indexed by subscripts as $\xspaceI$, $\tspaceJ$, etc.
All mappings, $\sigma$-fields and other quantities on these
spaces are indexed accordingly. We use superscripts $\xI^{(j)}$ to
denote elements of sequences or repetitive observations.
For any measure $\nu$, a superscript $\nu^{\ast}$ indicates the
corresponding outer measure.
Observations are generally assumed exchangeable.
Topological spaces are assumed to be Polish spaces, \ie
complete, separable and metrizable spaces, unless 
expressly stated otherwise. 
We refer to a measurable space as \emph{standard Borel}
if it is the Borel space generated by a Polish topology.
As the underlying spaces are Polish, all
conditional probabilities $P[X|\field]$ are assumed
to be regular conditional probabilities (probability kernels).

\pdfoutput=1

\section{Construction of Stochastic Processes}
\label{sec:stochproc}

We will briefly survey the 
construction of stochastic processes and introduce
some relevant definitions. The presentation assumes familiarity with
the terminology of projective limits, which is used here
in the sense of \citet{Bourbaki:1966,Bourbaki:1968,Bourbaki:2004}.
A more detailed summary of projective limits and pullbacks
is given in Appendix \ref{sec:background}.

\subsection{Projective Limit Notation}

Let $(D,\po)$ be a partially ordered, directed set.
\emph{We assume $D$ to be countable throughout.}
Let $\left< \xspaceI,\borelI,\fJI\right>_{\indI\po\indJ\in\indD}$,
or $\famD{\xspaceI,\borelI,\fJI}$ for short,
be a \emph{projective system} of topological measurable spaces
indexed by $D$. That is, $\xspaceI$ are topological spaces, $\borelI$ their
Borel $\sigma$-algebras, and $\fJI:\xspaceJ\rightarrow\xspaceI$
are continuous generalized projections;
the mappings are called generalized projections if they
satisfy
\begin{equation}
    \label{eq:def:canonical:map}
    \fII = \mbox{Id}_{\xspaceI} \quad\text{ and }\quad
    \fKI = \fKJ\circ\fKI \qquad\text{ whenever } I\po J\po K \;.
\end{equation}
Denote by $\xspaceD$ the projective limit space.
The mappings $\fJI$
induce a family of unique generalized projection
mappings $\fI:\xspaceD\rightarrow\xspaceI$.
The space $\xspaceD$ is endowed with the smallest topology
$\topD$ which makes all $\fI$ continuous. $\topD$
is called the \emph{projective limit topology}, and generates
the projective limit Borel $\sigma$-algebra $\borelD$.
A family $\famD{\PI}$ of probability measures on the spaces
$\xspaceI$ is called \emph{projective} if $\fJI(\PJ)=\PI$ whenever
$I\po J$. By the extension theorem of Kolmogorov and Bochner 
(App.~\ref{sec:background}, Theorem \ref{theorem:bochner}), any projective family
defines a unique probability measure
$\PD$ on $(\xspaceD,\borelD)$ which
satisfies $\PI=\fI(\PD)$ for all $I\in D$. 
We refer to this measure, also denoted $\PD=\plim\famD{\PI}$, 
as the \emph{projective
limit} of $\famD{\PI}$, and to the measures $\PI$ as the
\emph{marginals} of $\PD$. Intuitively, the measures $\PI$
are probability distributions on finite-dimensional spaces,
and $\PD$ is a joint distribution of a stochastic process
$\famD{\XI}$ on the infinite-dimensional space $\xspaceD$.

The projective limit space $\xspaceD$ is a 
subset of the product space $\prod_{I\in D}\xspaceI$. If
$\projectorI$ denotes the canonical projection onto $\xspaceI$
in the product space, the canonical mappings $\fI$ are the
restrictions $\fI=\projectorI\vert_{\xspaceD}$.
It is often useful to regard the elements $\xD$ of $\xspaceD$
as functions $\xD:D\rightarrow\cup_{I\in D}\xspaceI$, or
more precisely, as functions on $D$ taking values $x(I)\in\xspaceI$.
In the context of nonparametric Bayesian estimation, the
indices $I\in D$ may be thought of as covariates or sets of
covariates and the
function values $\xI=x(I)$ as measurements, if $\xspaceD$ represents
the observation space of the model. If $\xspaceD$ is a parameter
space, continuous real-valued functions $\xD$ may represent
regressors, set functions $\xD$ may represent density estimates, etc.

\subsection{Stochastic Processes}
\label{sec:stochproc:stochproc}

A stochastic process is in general a collection
$\famD{\XI}$ of random variables, indexed by an infinite
set $D$. Hence, if $\PD=\plim\famD{\PI}$ is a projective limit
measure with marginals $\PI$, the family $\famD{\XI}$ 
of random variables distributed according to the measures $\PI$
is a stochastic process indexed by $D$.
Conversely, any stochastic process
can in principle be regarded as the projective limit of
its marginals on suitably chosen subspaces. However,
constructions of stochastic processes as projective limits
have to address two fundamental technical problems:
\begin{enumerate}
\renewcommand{\labelenumi}{(\alph{enumi})}
\item {\emph{Uncountable index sets}.}
  An event $A\subset \xspaceD$ is measurable
  under $\PD$ only if it depends on an at most
  countable subset $D'\subset D$ of coordinates
  \citep[e.g.][Theorem 36.3]{Billingsley:1995}. In other words,
  unless $D$ is countable, singletons are not
  measurable in the projective limit space, and the
  projective limit measure $\PD$ is not useful
  for most applications.
\item {\emph{Infinitary properties of sample paths}.}
  If the spaces $\xspaceI$ in the projective system
  are finite-dimensional, the
  projective limit construction can only express properties
  of the random functions $\xD$ that are \emph{finitary},
  such as non-negativity or monotonicity of real-valued
  functions, or finite additivity of set functions. 
\end{enumerate}
Problem (a) means, for example, that projective limits 
can directly define
a useful measure on functions $\mathbb{Q}\rightarrow\mathbb{R}$,
but not on functions $\mathbb{R}\rightarrow\mathbb{R}$,
since the space $\mathbb{R}^{\mathbb{R}}$ of all functions
$\mathbb{R}\rightarrow\mathbb{R}$ has uncountable dimension.
Problem (b) implies, for example, that a projective limit construction
of random set functions can define a sample space consisting
of all charges (finitely additive probabilities),
but not a sample space containing exactly all 
probability measures, which would require the
projective limit to express countable additivity. 

Both problems (a) and (b) can be jointly
addressed in an elegant manner by means of pullbacks
under suitable functions.
Given a space $\xspace$, a measure space $(\yspace,\borelY,\nu)$
and a function $\inclusion:\xspace\rightarrow\yspace$,
the \emph{pullback} of $\nu$ under $\inclusion$ is
the measure $\tnu$ on $(\xspace,\inclusion^{-1}\borelY)$
satisfying $\inclusion(\tnu)=\nu$. The pullback measure
$\nu$ is uniquely defined whenever the image 
$\inclusion(\xspace)\subset\yspace$ has full outer
measure under $\nu$, that is if
$\nu^{\ast}(\inclusion(\xspace))=\nu(\yspace)$
-- see App.~\ref{sec:pullbacks} for more details.
The most common example of a pullback is
the restriction of a measure to a (possibly non-measurable)
subspace, in which case $\xspace\subset\yspace$ is an
arbitrary subset and $\inclusion:\xspace\hookrightarrow\yspace$
the canonical inclusion map. The $\sigma$-algebra
$\inclusion^{-1}\borelY$ is then precisely the
subspace $\sigma$-algebra $\borelY\cap\xspace$.
Hence, if $\nu$ is a probability measure on $\yspace$,
and if the subspace has outer measure $\nu^{\ast}(\xspace)=1$,
the pullback $\tnu$ exists and is the restriction of $\nu$
to $(\xspace,\borelY\cap\xspace).$

To construct stochastic processes, we will specifically
consider pullbacks under embedding maps. Let
$\phi:\txspace\rightarrow\xspace$ be a mapping
between topological spaces. Such a mapping is
called an \emph{embedding} if, regarded as a 
mapping onto its image, it is a homeomorphism. 
Analogously, we refer
to $\phi$ as a \emph{Borel embedding} if it constitutes
a Borel isomorphism of its domain and its image
$(\tImage,\borel(\xspace)\cap\tImage)$.
A definition of a stochastic process suitable for our
questions in Bayesian nonparametrics is the following:
\begin{definition}
  \label{def:cr:stoch:proc}
  Let $(\txspace,\borel(\txspace),\tP)$ be a topological measure
  space and $\famD{\xspaceI,\borelI,\PI}$ a projective system of
  standard Borel spaces with countable, directed index set $D$. 
  Then $\tP$ is called a \emph{countably representable 
    stochastic process} if it is the pullback of
  the projective limit measure $\PD:=\plim\famD{\PI}$
  under a Borel embedding
  $\rest:\txspace\rightarrow\tImage\subset\xspaceD$.
\end{definition}
To be asymptotically identifiable, a model can have at most
a countable number of degrees of freedom, which motivates
the restriction to sample paths of countable complexity
implicit in Definition \ref{def:cr:stoch:proc}:
The indices $I\in D$ of a projective limit
can be thought of as dimensions or
degrees of freedom. Hence, the sample space 
$\txspace$ of a stochastic
process with countably many degrees of freedom can be embedded
into a suitably chosen projective limit space $\xspaceD$
with countable index set.

The special case in which $\nu$
is a projective
limit measure on an \emph{uncountable} product space
$\yspace:=\xspaceD$, constructed from Euclidean spaces
$\xspaceI=\mathbb{R}^{\indI}$, and $\xspace$ is \eg the
subset of continuous functions, is known in stochastic
process theory as ``Doob's separability theorem''.
In this case, the pullback $\tnu$ is called
a ``separable modification'' of $\nu$ 
\citep{Doob:1953}. The index set $D$ is the set
of all finite subsets of the ``separant'', a
dense countable subset of $\mathbb{R}_+$. See also
\citep[][Chapter 38]{Billingsley:1995}.

The intuition that sample paths of $\tP$ (the elements
of $\txspace$) are uniquely represented by their
embeddings into $\xspaceD$
can be helpful in establishing that a given mapping $\phi$
is indeed a Borel embedding:
Suppose that a measurable map $\rest$ is given. As a 
mapping onto its image, it is trivially surjective, so
what remains to be established for Borel isomorphy 
is the existence of a measurable inverse.
If the elements of $\txspace$ are uniquely represented
by their embeddings, then $\rest$
is injective. 
In most settings, the mapping $\rest$ can be directly
derived from a suitable representation result, such as
the representation of continuous functions by their
values on countable subsets as mentioned above, or
the representation of measures by their values on a 
generating algebra of sets (by Carath\'{e}odory's extension
theorem). If additionally both $\txspace$ and $\tImage$ are
standard Borel spaces, Borel isomorphy follows automatically,
since measurable bijections between standard Borel spaces
are bimeasurable \citep[][Theorem A1.3]{Kallenberg:2001}.

\begin{example}[Dirichlet process]
  \label{example:DP}
  Suppose that $\tP$ is a Dirichlet process $\DP{\alpha G_0}$ over
  a standard Borel  space $(V,\borelV)$. 
  The spaces $\xspaceI$
  can be chosen as finite-dimensional simplices $\simpI\subset\mathbb{R}^{\indI}$,
  indexed by measurable partitions $I=(A_1,\dots,A_{|I|})$ of
  the space $V$. The marginals $\PI(\XI)$ are Dirichlet
  distributions on the simplices.
  The projective limit is the space of all 
  charges defined on a 
  specific countable algebra $\mathcal{Q}\subset\borelV$ which generates $\borelV$.
  The space $\txspace$ is the space of all probability
  measures on $\borelV$, and its image $\tImage=\phi(\txspace)$
  is the set of probability measures on the subalgebra $\mathcal{Q}$.
  For a given measure $\tx$ on $\borelV$, the image $\phi(\tx)$
  is the restriction of $\tx$ to $\mathcal{Q}$. By the
  Carath\'{e}odory extension theorem,
  $\phi$ is injective. Whether $\PD$ admits a pullback under
  $\phi$ depends on the parametrization of the marginals:
  If $G_0$ is a charge on $\mathcal{Q}$, and each Dirichlet marginal
  has parameter $\alpha\cdot\fI(G_0)$ for some fixed $\alpha>0$, the
  Dirichlet distributions form a projective family. The projective
  limit satisfies $\PD^{\ast}(\tImage)=1$ if and only
  if $G_0$ is countably additive. Sec.~\ref{sec:examples:DP} revisits
  this example in detail.
\end{example}

\begin{example}[Gaussian Process]
  \label{example:GP:C}
  To obtain a Gaussian process measure
  on the set $\txspace:=\cfspace(\mathbb{R}_+,\mathbb{R})$ of
  continuous functions $\mathbb{R}_+\rightarrow\mathbb{R}$,
  a projective limit is constructed as follows:
  Choose $D$ as the set of all finite subsets $I$ of $\mathbb{Q}_+$,
  ordered by inclusion, and define
  $\xspaceI:=\prod_{i\in I}\mathbb{R}$.
  Let $\fJI:=\projectorJI$ be the coordinate projections in
  Euclidean space, and $\famD{\PI}$ a projective family of
  multivariate Gaussian distributions. The projective limit
  space is $\xspaceD=\mathbb{R}^{\mathbb{Q}_+}$, and the projective
  limit measure $\PD$ can be regarded as a discrete-time 
  Gaussian process indexed by $\mathbb{Q}_+$.
  We embed $\txspace$ into $\mathbb{Q}_+$ by means of the
  restriction map $\rest:\tx\mapsto\tx\vert_{\mathbb{Q}_+}$.
  The mapping $\rest$ is a Borel isomorphism as required
  in Definition \ref{def:cr:stoch:proc}: As a canonical inclusion map,
  $\rest$ is continuous and hence measurable. 
  Since the representation
  of $\tx$ by its restriction is unique, $\rest$ is injective.
  The $\sigma$-algebra $\rest^{-1}\borelD$ induced
  by $\rest$ on $\cfspace(\mathbb{R}_+,\mathbb{R})$
  coincides with the Borel $\sigma$-algebra generated by
  the topology of compact convergence \citep[][Section 454O]{Fremlin:MT}.
  Hence, $\txspace$ is standard Borel, and $\rest$
  bimeasurable. 
  The requirement $\PD^{\ast}(\txspace)=1$ for the existence
  of the pullback measure is \emph{not} generally satisfied
  for arbitrary Gaussian marginals $\PI$. It can, however,
  be related to the parameters of the marginals. A prototypical
  result is Kolmogorov's continuity theorem 
  \citep[][Theorem 39.3]{Bauer:1996}:
  If the expectation under $\PD$ satisfies
  $\mean{|X_i-X_j|^{\alpha}}\leq\gamma|i-j|^{\beta}$ for all $i,j\in\mathbb{Q}_+$
  and any fixed $\alpha,\beta,\gamma\in\mathbb{R}_{>0}$,
  then $\PD^{\ast}(\cfspace(\mathbb{R}_+,\mathbb{R}))=1$.
  An example to the contrary is obtained for marginals satisfying
  $\mbox{Cov}[X_i,X_j]=\delta_{ij}$. The resulting
  Gaussian white noise process is almost surely discontinuous,
  and hence $\PD^{\ast}(\txspace)\neq 1$.
\end{example}

\pdfoutput=1

\section{Projective Limits of Conditional Probabilities}
\label{sec:conditionals}

In this section, we apply the projective limit approach to
conditional probabilities. By means of Theorem 
\ref{theorem:projlim:conditionals} below,
a conditional probability on an infinite-dimensional space
can be assembled as a projective limit of
conditional probabilities on finite-dimensional spaces,
in a similar manner as a probability measure
can be specified as a projective limit by means of 
the Kolmogorov-Bochner extension theorem.

\subsection{Construction Results}

Let $\famD{\xspaceI,\borelI,\fJI}$ be a projective system of 
standard Borel spaces.
For each $I\in D$, let $\PI[\XI|\fieldI]$ be a regular conditional
probability on $(\xspaceI,\borelI)$. More precisely, 
$\XI:\abstspace\rightarrow\xspaceI$
is a random variable,
$\fieldI\subset\abstfield$ is a $\sigma$-subalgebra on the abstract
probability space $\abstspace$, and 
$\PI[\,.\,|\fieldI](\,.\,):\borelI\times\abstspace\rightarrow[0,1]$
is a probability kernel.

The projections $\fJI$ immediately generalize
from probability measures to conditional probabilities by means of
\begin{equation}
  (\fJI\PJ)[\XJ\in .\,|\fieldJ] := \PI[\XI\in\fJI^{-1}\,.\,|\fieldJ] \;.
\end{equation}
The projector acts only on the first argument of the probability kernel.
To generalize the notion of a projective family, the second argument
has to be taken into account as well:
Consider a parametric family $\PI[\XI|\ThetaI]$, \ie each $\fieldI$
is generated by a parameter random variable $\ThetaI$.
Typically, if $\ThetaJ$ parametrizes a
high-dimensional random variable $\XJ$ and 
$\ThetaI$ a lower-dimensional
variable $\XI$, we would assume the information contained in
$\ThetaI$ to be a subset of the information contained in $\ThetaJ$.
The concept can be expressed in very general terms by assuming that
the $\sigma$-algebras $\fieldI$ are ordered in accordance with the
index set, \ie $\fieldI\subset\fieldJ$ whenever $I\po J$.
In analogy to the index set, we refer to such an ordered 
family of $\sigma$-algebras as \emph{directed}.
\begin{definition}[Projective family of conditional probabilities]
\label{def:cond:proj:family}
Let $\famD{\fieldI}$ be a directed family of $\sigma$-algebras.
A family $\famD{\PI[\XI|\fieldI]}$ of probability kernels on the
the projective system $\famD{\xspaceI,\borelI,\fJI}$ is called
\emph{projective} if 
\begin{equation}
  \label{eq:def:cond:proj:family}
  (\fJI\PJ)[\,.\,|\fieldJ] \eqae \PI[\,.\,|\fieldI]\qquad\text{ whenever }
  I\po J\;.
\end{equation}
\end{definition}
Projectivity of conditionals is a stronger condition than projectivity
of measures: We have 
$\P(A)=\int_{\abstspace}\P[A|\field](\omega)d\abstmeasure(\omega)$
for any $\field\subset\abstfield$,
and hence
$\PJ[\fJI^{-1}\AI|\fieldJ]
\equae\PI[\AI|\fieldI](\omega)$
implies
$\PJ(\fJI^{-1}\AI)
=\PI(\AI)$.
Therefore, projective conditionals imply projective measures,
but the converse only holds under additional conditions
(cf Lemma \ref{lemma:cond:proj:criterion:1}). 
If the conditional distributions of random variables
$\XI$ are projective given one directed family of
$\sigma$-algebras, the same may be not true for another
family, so the conditional projector is effectively parametrized
by the family $\famD{\fieldI}$.



\begin{theorem}[Projective limits of conditional probabilities]
  \label{theorem:projlim:conditionals}
  Let $E$ be a countable directed set.
  Let $\famD{\PI[\XI|\fieldI]}$ be a projective family of
  probability kernels on a projective system
  $\famD{\xspaceI,\borelI,\fJI}$ of Polish measurable spaces.
  Then there exists a unique (up to
  equivalence) probability kernel, denoted $\PD[\,.\,|\fieldD]$,
  which satisfies
  \begin{equation}
    \label{eq:theorem:bochner:marginals}
    (\fI\PD)[\,.\,|\fieldD] \eqae \PI[\,.\,|\fieldI] \qquad\text{ for
      all } I\in D\;,
  \end{equation}
  and is measurable with respect to $\fieldD:=\sigma(\fieldI ; I\in D)$.
\end{theorem}

As the proof below shows, $\PD[\,.\,|\fieldD]$ can be regarded as
the projective limit of the measurable, measure-valued functions
$\omega\mapsto\PI[\,.\,|\fieldI](\omega)$.
In analogy to probability measures, we refer to the conditionals
$\PI[\XI|\fieldI]$ as the \emph{marginal conditional probabilities}
of $\PD[\XD|\fieldD]$, or \emph{marginals} for short.

\begin{proof}
The proof relies on the simple fact that measurability of mappings
is preserved under projective limits (as is continuity
\citep[][I.4.4]{Bourbaki:1966}):
\def\wI{w_{\indI}}
\def\wD{w_{\indD}}
\begin{lemma}
  \label{lemma:projlim:measurable:maps}
  Let $(\abstspace,\abstfield)$ be a measurable space,
  $\famD{\xspaceI,\borelI,\fJI}$ a projective family of measurable
  spaces with projective limit $(\xspaceD,\borelD)$, and
  $\famD{\wI:\abstspace\rightarrow\xspaceI}$ a projective family of
  measurable mappings. Then the projective limit $\wD:=\plim{\wI}$ is
  a measurable mapping $\abstspace\rightarrow\xspaceD$.
\end{lemma}
\begin{proof}[Proof (Lemma \ref{lemma:projlim:measurable:maps})]
  Since $\famD{\wI}$ is projective, 
  $\wI\circ\fI=\fI\circ\wD$. By measurability of
  $\wI$ and $\fI$, the composition $\fI\circ\wD$
  is $\abstfield$-$\borelI$-measurable for all $I\in D$.
  Since the canonical mappings $\fI$ generate $\borelD$,
  $\wD$ is $\abstfield$-$\borelD$-measurable.
\end{proof}

\def\Ncomp{N^{\mbox{\tiny C}}}
\def\PDrest{\PD^{\ind{$\backslash$N}}}

  The regular conditional probabilities $\PI[\XI|\fieldI]$
  can be regarded as a family of random measures, \ie as
  measurable mappings $\PI:\abstspace\rightarrow\pMeas(\xspaceI)$
  defined by $\omega\mapsto\PI[\XI|\fieldI](\omega)$.
  To prove Theorem \ref{theorem:projlim:conditionals}, we argue
  that this family is projective (in the sense of App.~\ref{sec:background},
  Lemma \ref{lemma:projlim:mappings}),
  with the desired conditional probability
  $\PD[\XD|\fieldI]$ as its projective limit. However, we have to account
  for the fact projectivity
  of the mappings holds only almost everywhere.

  Denote by $\pMeas(\xspaceI)$ the set of probability measures
  on $\xspaceI$. The continuous mappings $\fJI$ induce, by means
  of $\PJ\mapsto\fJI(\PJ)$, a continuous projection 
  $\fJI:\pMeas(\xspaceJ)\rightarrow\pMeas(\xspaceI)$.
  With respect to these projectors, the measurable mappings
  $\PI:\abstspace\rightarrow\pMeas(\xspaceI)$ are projective almost everywhere: 
  For any pair $I\po J$ of indices, \eqref{eq:def:cond:proj:family} holds up to a null set
  $N_{\indJ\indI}\subset\abstspace$ of exceptions. Write
  $N:=\cup_{I\po J}N_{\indJ\indI}$ for the aggregate null set, $\Ncomp:=\abstspace\setminus N$
  for its complement.
  The restricted mappings $\PI\vert_{\Ncomp}:\Ncomp\rightarrow\pMeas(\xspaceI)$
  form a projective family of $\fieldD\cap\Ncomp$-measurable mappings, and by
  Lemma \ref{lemma:projlim:measurable:maps} have a unique, measurable
  projective limit $\PDrest:\Ncomp\rightarrow\pMeas(\xspaceD)$.
  This mapping satisfies
  \begin{equation}
    (\fI\PDrest)(\omega)=\fI(\PDrest(\omega))=\PI(\omega) \qquad\text{ for all }\omega\in\Ncomp\;.
  \end{equation}
  The first identity is due to the definition of projective limit mappings;
  the second follows by observing that, for any $\omega\in\Ncomp$, 
  $\famD{\PI(\omega)}$ is a projective family of probability
  measures with projective limit measure $\PDrest(\omega)$.

  As a countable projective limit of Polish spaces, $\xspaceD$ is Polish, and so is
  $M(\xspaceD)$. Therefore, the $\fieldD\cap\Ncomp$-measurable function
  $\PDrest:\Ncomp\rightarrow\pMeas(\xspaceD)$ has an extension to a measurable
  function $\PD:\Omega\rightarrow\pMeas(\xspaceD)$ \citep[][Theorem 12.2]{Kechris:1995}. 
  This function
  $\PD(\omega)=:\PD[\XD|\fieldD](\omega)$ is a regular conditional probability on $\xspaceD$, and
  satisfies \eqref{eq:theorem:bochner:marginals} $\abstmeasure$-almost everywhere.
\end{proof}

Like projective limits, pullbacks generalize from measures
to conditional probabilities. 
\begin{proposition}[Pullback of regular conditional probabilities]
  \label{lemma:pullback:conditionals}
  Let $P[X|\field]$ be a regular conditional probability on 
  a standard Borel space $\xspace$. Let $\txspace$ be a
  Hausdorff space, $\phi:\txspace\rightarrow\xspace$
  injective, and $\tborel:=\phi^{-1}\borel(\xspace)$ the
  induced $\sigma$-algebra on $\txspace$. Denote by $\tOmega\subset\Omega$
  the set of all $\omega$ satisfying $P^{\ast}[\phi(\txspace)|\field](\omega)=1$.
  Then $\nu(A,\omega):=P[\phi(A)|\field](\omega)$ is a
  probability kernel on $\txspace$, and can be regarded
  as a regular conditional probability of the random
  variable $\tX:=\phi^{-1}\circ X\vert_{\tOmega}$, given $\field\cap\tOmega$.
\end{proposition}
Clearly, $\tOmega$ may be empty. A pullback construction of a model
will therefore typically involve a result characterizing
either $\tOmega$ or a subset of $\tOmega$. The characterization is 
usually expressed as the image of $\tOmega$ under a suitable parameter
random variable, \ie as a result describing a set of ``parameter values''
for which the model concentrates on $\txspace$. 
An example of such a characterization is the Kolmogorov continuity theorem
mentioned in Example \ref{example:GP:C}: The Gaussian process in the example can be parametrized
by its mean and covariance functions, and the theorem specifies
a subset of parameter for which the pullback exists.
Lemma
\ref{lemma:examples:DP:concentration:prior},
\ref{lemma:examples:DP:concentration:likelihood}
and 
\ref{lemma:example:cayley:concentration}
in Sec.~\ref{sec:examples} are further examples of such results.
\begin{proof}
  Since $\tborel$ is the $\sigma$-algebra induced
  by $\phi$ and $\phi$ is injective, the inverse $\phi^{-1}$ is
  automatically measurable with respect to $\borel(\xspace)\cap\phi(\txspace)$,
  so the restriction of the mapping
  $\phi^{-1}\circ X$ is indeed a valid $\txspace$-valued random variable 
  on $\tOmega$. The result follows by a simple point-wise application
  of pullbacks to the measures $P[\,.\,|\field](\omega)$ for
  $\omega\in\abstspace$.
\end{proof}
The combination of Theorem \ref{theorem:projlim:conditionals} 
and Lemma \ref{lemma:pullback:conditionals} results in a two-stage
approach to the construction of regular conditional probabilities,
analogous to the two-stage construction of stochastic processes
in the sense of Definition \ref{def:cr:stoch:proc}:
First construct a suitable projective limit 
$\PD[\XD|\fieldD]$, and then pull back to a (possibly
non-measurable) subspace $\txspace\subset\xspaceD$,
or to a space $\txspace$ embedded into $\xspaceD$
by a Borel embedding $\phi$.

Both steps can be combined into a single step under an
additional assumption --  
namely that the embedding of $\txspace$, \ie the image $\phi(\txspace)$,
is actually
measurable in $\xspaceD$. The extension result obtained for this
case can be regarded as a conditional probability analogue of the
well-known projective limit theorem of Prokhorov 
\citep[][IX.4.2]{Bourbaki:2004}, just as Theorem
\ref{theorem:projlim:conditionals} is analogous to the
extension theorems of Kolmogorov and Bochner.
\begin{corollary}[Prokhorov extension]
  \label{corollary:prokhorov:extension}
  Let $\famD{\xspaceI,\borelI,\fJI}$ be a countably indexed projective
  system of Polish measurable spaces, $\txspace$ a Hausdorff space,
  $\phi:\txspace\rightarrow\xspaceD$ continuous and injective, and
  require $\phi(\txspace)\in\borelD$.
  Let $\famD{\PI[\XI|\fieldI]}$ be a projective family of
  probability kernels on $\borelI\times\abstspace$. Define $\tOmega$
  to be the subset $\tOmega\subset\abstspace$ of all $\omega$
  for which the family of measures $\famD{\PI[\,.\,|\fieldI](\omega)}$
  satisfies the following ``Prokhorov condition'':\\
  For all $\varepsilon>0$, there is a compact set $K\subset\txspace$
  such that
  \begin{equation}
    \PI[\phiI K|\fieldI](\omega) > 1 - \varepsilon \qquad\qquad\text{
      for all }I\in D\;.
  \end{equation}
  Then there is a unique (up to equivalence) probability kernel 
  $\tP[\,.\,|\tilde{\field}_{\indD}](\omega)$ on
  $\borel(\txspace)\times\tOmega$ with the projective family as its
  marginals, \ie $\phiI\tP[\,.\,|\fieldD]\eqae\PI[\,.\,|\fieldI]$.
  This probability kernel
  \begin{enumerate}
  \item is a Radon measure for each $\omega\in\tOmega$;
  \item is the pullback of 
    $\PD[\,.\,|\fieldD]=\plim\famD{\PI[\,.\,|\fieldI]}$ under $\phi$,
    and hence a conditional probability given $\tilde{\field}_{\indD}
    =\fieldD\cap\tOmega$.
  \end{enumerate}
\end{corollary}
In the following sections, we will derive a number of results on
how certain statistical properties of conditional models are preserved
under projective limits and pullbacks. For conditional probabilities
constructed by means of the Corollary, statement (2)
makes all these results immediately applicable,
since the constructed probability kernel 
$\tP[\,.\,|\tilde{\field}_{\indD}]$
can effectively be decomposed into the projective limit 
$\PD[\,.\,|\fieldD]$ and a subsequent pullback.
\begin{proof}
  For almost all $\omega\in\tOmega$, the measures
  $\PI[\,.\,|\fieldI](\omega)$ form a projective family and satisfy
  the Prokhorov condition. Since the spaces $\xspaceI$ are Polish,
  each of these measures is a Radon measure.
  By Prokhorov's theorem \citep[][IX.4.2]{Bourbaki:2004}, there is
  a unique Radon probability measure $\nu_{\omega}$ on $\txspace$ satisfying
  $\phiI(\nu_{\omega})=\PI[\,.\,|\fieldI](\omega)$.
  By the Kolmogorov-Bochner extension theorem 
  (App.~\ref{sec:background}, Theorem \ref{theorem:bochner}), there is also a
  unique projective limit probability kernel $\PD[\,.\,|\fieldD]$
  on $\xspaceD=\plim\famD{\xspaceI}$. 
  Since $\phi=\plim\famD{\phiI}$, we have 
  $\PD[\,.\,|\fieldD](\omega)=\phi(\nu_{\omega})$ for almost all $\omega\in\tOmega$.
  The image $\phi(\txspace)$ is measurable, and so
  $\PD^{\ast}[\phi(\txspace)|\fieldD](\omega)=\nu_{\omega}(\phi^{-1}\phi\txspace)=\nu_{\omega}(\txspace)$.
  Therefore, the pullback under $\phi$ exists, and by uniqueness has to
  coincide with $\nu_{\omega}$ almost everywhere.
\end{proof}
The induced conditional probabilities $\tP[\,.\,|\fieldD]$
on $\txspace$ are regular, since measurability
in $\omega$ carries over from $\xspaceD$ under the pullback.
This is remarkable in so far as virtually no requirements
are imposed upon the space $\txspace$ --
in particular, the topology of $\txspace$ need not admit a countable
subbase -- and conditional probabilities on
$\txspace$ need not be regular in general.
In other words, much as the Radon regularity of measures 
on a space which supports non-Radon probability measures
is induced by the marginals, so is regularity of the
conditional.

\subsection{Criteria for Projectivity}

To construct a conditional stochastic process
$\PD[\XD|\fieldD]$ by means of Theorem 
\ref{theorem:projlim:conditionals} 
will in practice require proof
that a given family of conditional 
probabilities is projective. The following
two results provide applicable criteria.

\begin{lemma}[Criterion 1]
  \label{lemma:cond:proj:criterion:1}
  Let the random variables $\XI$
  satisfy $\fJI\XJ=\XI$, and let 
  $\famD{\fieldI}$ be a directed family of
  $\sigma$-algebras.
  Then the family $\famD{\PI[\XI|\fieldI]}$
  is projective if and only if
  the random variables satisfy the conditional
  independence relations
  \begin{equation}
    \label{eq:criterion:1}
    \XI\condind_{\fieldI}\fieldJ \qquad\text{ for all }
    I\po J\;.
  \end{equation}
\end{lemma}
\begin{proof}
  By the properties of conditional independence,  
  \begin{equation}
    \label{eq:proof:lemma:criterion:1}
    \XI\condind_{\fieldI}\fieldJ
    \quad\Leftrightarrow\quad
    \abstmeasure[A|\fieldI,\fieldJ]\eqae\abstmeasure[A|\fieldI]
    \quad\text{ for all } A\in\sigma(\XI)\;.
  \end{equation}
  See \citep[][Proposition 6.6]{Kallenberg:2001}.
  Application to the definition of projectivity yields
  \begin{equation}
      \PJ[\fJI^{-1}\AI|\fieldJ]
      \quad\stackrel{\fJI\XJ=\XI}{=}
      \abstmeasure[\XI^{-1}\AI|\fieldJ]
      \quad\stackrel{\eqref{eq:proof:lemma:criterion:1}}{=}
      \PI[\AI|\fieldI]\;.
  \end{equation}
\end{proof}
We recall that projectivity
of conditional probabilities $\PI[\XI|\ThetaI]$ as in
\eqref{eq:def:cond:proj:family} implies projectivity of the corresponding
unconditional measures $\PI=\XI(\abstmeasure)$.
Lemma \ref{lemma:cond:proj:criterion:1} gives a necessary and
sufficient condition for the converse to hold as well: 
If the $\sigma$-algebras $\fieldI$ are
generated by parameter variables
$\ThetaI$, \eqref{eq:criterion:1} takes the form
$\XI\condind_{\ThetaI}\ThetaJ$.
For a fixed $I$, the criterion demands that 
-- given full knowledge of $\ThetaI$ --
information about the parameters corresponding to any other dimensions
will not change our mind about $\XI$. If this is true
for any $I$, the family is conditionally projective.
The lemma implies a similar result by \citet[][IV, 3.1]{Lauritzen:1988}
on sufficient statistics: Since \eqref{eq:criterion:1} is a necessary condition,
any sufficient statistics $\famD{\SI}$ satisfy $\XI\condind_{\SI}\SJ$
if the family of models is known to be projective.

\def\nuI{\nu_{\indI}}
\def\nui{\nu_{\indi}}

In practice, a candidate family of finite-dimensional conditionals
can be expected to be defined by densities, with respect to some
family $\famD{\nuI}$ of carrier measures.
The next criterion addresses the special case where the projective system consists
of product spaces $\xspaceI=\prod_{i\in I}\xspace_{\indi}$
as in Example \ref{example:GP:C}, and hence $\fJI=\projectorJI$.
The carrier measures are then typically product measures, and
proving that the family is projective 
involves an application of Fubini's theorem. The following criterion
makes this step generic.
\begin{lemma}[Criterion 2]
  \label{lemma:cond:proj:criterion:2}
  Let $\famD{\PI[\XI|\ThetaI]}$ be a
  family of conditional probabilities on a projective system
  $\famD{\prod_{i\in I}\xspace_{\indi},\otimes_{i\in I}\borel_{\indi},\projectorJI}$,
  where each $\xspace_{\indI}$ is Polish. Require:
  (1) For all $I\in D$, the conditional density $p_{\indI}$ of
  $\PI[\XI|\ThetaI]$ with respect to a
  carrier measure $\nuI$ on $\xspaceI$ exists.
  (2) The carrier measures are product measures
    $\nuI=\otimes_{i\in I}\nui$.
  Then the family $\PI[\XI|\ThetaI]$ of conditionals
  is projective if and only if
  \begin{equation}
    \label{eq:lemma:criterion:2:condition}
    \int_{\xspace_{\indJI}}p_{\indJ}(x_{\indJ}|\theta_{\indJ})d\nu_{\indJI}(x_{\indJI})
    =
    p_{\indI}(x_{\indI}|\projectorJI\theta_{\indJ})
    \qquad\text{ whenever } I\po J\;.
  \end{equation}
\end{lemma}
The use of $J\setminus I$ as an index 
is justified by the fact that $D$ consists
of all finite subsets of a given set, and is ordered by inclusion.
Therefore, $I\po J$ implies $J\setminus I\in D$.

\begin{proof}
  First suppose condition \eqref{eq:lemma:criterion:2:condition}
  is satisfied. 
  Denote by $p_{\indI}(x_{\indI}|\theta_{\indI})$ the conditional
  density of $\PI[\XI|\ThetaI]$. By Fubini's theorem, 
  \begin{equation}
    \begin{split}
    \int_{\projectorJI^{-1}\AI}
    p_{\indJ}(x_{\indJ}|\theta_{\indJ})d\nu_{\indJ}(x_{\indJ})
    =&
    \int_{\AI}(
    \int_{\xspace_{\indJI}}
    p_{\indJ}(x_{\indI},x_{\indJI}|
    \theta_{\indJ})d\nu_{\indJI}(x_{\indJI}))d\nu_{\indJ}(x_{\indJ})\\
    =&
    \int_{\AI}
    p_{\indI}(x_{\indI}|
    \projectorJI\theta_{\indJ})d\nu_{\indJ}(x_{\indJ})
    \end{split}
  \end{equation}
  for all $\AI\in\borelI$. Hence,
  $\PJ[\projectorJI^{-1}\AI|\ThetaJ=\theta_{\indJ}]
  =
  \PI[\AI|\ThetaI=\projectorJI\theta_{\indJ}]$ for all $\theta_{\indJ}$ up to
  a null set, 
  which establishes the ``if'' implication.
  Conversely, assume that the family is projective.
  Abbreviate $a(x_{\indI},\theta_{\indJ}):=
  \int_{\xspace_{\indJI}}p_{\indJ}(x_{\indJ}|\theta_{\indJ})d\nu_{\indJI}(x_{\indJI})$.
  Then $\PJ[\projectorJI^{-1}\AI|\ThetaJ=\thetaJ]=   \int_{\AI}
    a(x_{\indI},\theta_{\indJ})
    d\nu_{\indI}(x_{\indI})$, and 
  for all $\AI\in\borelI$,
  \begin{equation}
    \int_{\AI}
    a(x_{\indI},\theta_{\indJ})
    d\nu_{\indI}(x_{\indI})
    =
    \int_{\AI}
    p_{\indI}(x_{\indI}|\theta_{\indI})
    d\nu_{\indI}(x_{\indI})
    =
    \int_{\AI}p(x_{\indI}|\projectorJI\theta_{\indJ})d\nu_{\indI}(x_{\indI})
  \end{equation}
  The first identity is simply projectivity, the second
  one follows from the fact that $a(\,.\,,\theta_{\indJ})$ is
  $\borelI$-measurable by Tonelli's theorem.
  Since $a$ and $p$ integrate
  identically over all $\AI$ and are
  $\borelI$-measurable,
  $a(\,.\,,\theta_{\indJ})=p_{\indI}(\,.\,|\projectorJI\theta_{\indJ})$
  holds $\nu_{\indI}$-a.s.
\end{proof}

\pdfoutput=1

\section{Application to Bayesian Models}
\label{sec:bayesian}
 
The results of the previous section provide
the formal
means of defining projective limits of Bayesian models,
since a Bayesian model is completely defined by a pair of conditional
probabilities. Combination of such projective limits with pullbacks
under Borel embeddings allows us to represent nonparametric 
Bayesian models by projective families of finite-dimensional Bayesian models.
Since the term ``parametric model'' is often associated with
finite-dimensional or dominated models,
we will instead use the term ``parametrized'' to describe
a statistical model indexed by a parameter, regardless of whether the dimension
of the parameter is finite or infinite.

\subsection{Parametrized and Bayesian Models}

We briefly recall the formal notion of model and parameter;
a detailed discussion is given by \citet[][Ch.\ 1.5.5]{Schervish:1995}.
Let $X:\abstspace\rightarrow\xspace$ be a random variable with
values in a Polish space $\xspace$, such that
$P^{\infty}=X^{\infty}(\abstmeasure)$
is exchangeable. Let $\pMeas(\xspace)$ be the
set of probability measures on $\xspace$, and denote by
$F:\xspace^{\infty}\rightarrow\pMeas(\xspace)$ the mapping
induced by the empirical measure. Let 
$\Psi$ be a parametric index, \ie a bimeasurable mapping
from the image $(F\circ X^{\infty})(\Omega)\subset\pMeas(\xspace)$
onto a measurable space $(\tspace,\borelT)$.
Then the derived random variable $\Theta:=\Psi\circ F\circ X^{\infty}$
is called a \emph{parameter}, and we call the regular conditional
probability $P[X|\Theta]$ a \emph{parametrized model}.
In summary,
\begin{diagram}[LaTeXeqno,small]
  \label{diagram:parametric:model}
  \abstspace &\rTo^{\quad X^{\infty}\quad } & \xspace^{\infty} & \rTo^{\quad F\quad} &
  \pMeas(\xspace)\supset\model &\rTo^{\quad \Psi\quad } & \tspace \;,
\end{diagram}
where the set
$\model:=\lbrace P[X|\Theta=\theta] \,|\, \theta\in\tspace\rbrace$
is the model $P[X|\Theta]$ regarded as a family of measures.
The assumption that $\xspace$ is Polish guarantees both the existence of
regular conditional probabilities on $\xspace$ and the validity
of de Finetti's theorem \citep[][Theorem 11.10]{Kallenberg:2001}. The theorem
in turn implies a law of large numbers, which
guarantees convergence 
$\lim_n F_n(X^{n})\rightarrow X(\abstmeasure)$ of the empirical
measure in the \wstar topology
on $\pMeas(\xspace)$, and hence ensures that $F$ is well-defined.

The parameter random variable
$\Theta$ induces an image measure
$\PTheta=\Theta(\abstmeasure)$
on the parameter space $(\tspace,\borelT)$.
For any given abstract random event $\omega\in\Omega$, the
corresponding value $\theta=\Theta(\omega)$ of the parameter
is completely determined by $X^{\infty}(\omega)$, as the image under
$\Psi\circ F$. The partial information about $\Theta(\omega)$
contained in a finite sample $X^n(\omega)=x^n$ can be conditioned
on as $\PTheta[\Theta|X^n=x^n]$. Under suitable
conditions the actual value $\theta=\Theta(\omega)$ is asymptotically
recovered as 
$\PTheta[\Theta|X^n=x^n]\xrightarrow{n\rightarrow\infty}\delta_{\theta}$.
In this context, $\PTheta(\Theta)$ is referred to as a \emph{prior
distribution}, $P[X|\Theta]$ as a \emph{sampling model} or \emph{likelihood},
and $\PTheta[\Theta|X]$ as the \emph{posterior} under observation $X$.
Additionally, the prior can be represented as a parametrized model
$\PTheta[\Theta|Y=y]$, where $Y$ is a
\emph{hyperparameter}. We refer to the whole system summarily as
the \emph{Bayesian model} defined by $P[X|\Theta]$ and $\PTheta[\Theta|Y]$.
For our purposes, it is sufficient to assume that the prior
is the ``true'' prior, \ie the actual image measure under $\Theta$.
If the sampling model is dominated,
Bayes' theorem is applicable, and the posterior can be represented
by the density $\frac{p(x|\theta)}{p(x)}$ with respect to the prior.
For undominated models, notably the Dirichlet process,
some alternative to Bayes' theorem is required. In Bayesian nonparametrics, this
alternative is 
usually conjugacy (Sec.~\ref{sec:conjugacy}).

\subsection{Application of Projective Limits}
\label{sec:bayesian:projlim}

\def\PsiI{\Psi_{\indI}}
\def\PsiJ{\Psi_{\indJ}}
\def\tPsi{\tilde{\Psi}}
\def\thetaI{\theta_{\indI}}
\def\thetaJ{\theta_{\indJ}}
Suppose that, for a projective system of sample spaces
$\famD{\xspaceI,\fJI}$, a parametrized model is given on each
space: Each object in
\eqref{diagram:parametric:model}, except
for the abstract probability space $\abstspace$,
is equipped with an index $I$.
\begin{diagram}[LaTeXeqno,small]
  \label{diagram:parametric:model:projlim}
  & & &\xspaceJ^{\infty} & \rTo^{\quad F_{\indJ}\quad} &
  \pMeas(\xspaceJ)\supset\model_{\indJ} &\rTo^{\quad \Psi_{\indJ} \quad } & \tspaceJ\\
  \abstspace & & \ruTo(3,1)^{\XJ^{\infty}} & \dTo_{\fJI} &&
  \dTo_{\fJI} && \dDotsto_{\gJI} \\
  & \rdTo(3,1)_{\XI^{\infty}} 
  & & \xspaceI^{\infty} & \rTo^{F_{\indI}} &
  \pMeas(\xspaceI)\supset\model_{\indI} &\rTo^{\Psi_{\indI}} & \tspaceI
\end{diagram}
The mappings $\fJI:\xspaceJ\rightarrow\xspaceI$
induce projections $\fJI:\xspaceJ^{\infty}\rightarrow\xspaceI^{\infty}$ and
$\fJI:\pMeas(\xspaceJ)\rightarrow\pMeas(\xspaceI)$. If $\fJI\model_{\indJ}=\model_{\indI}$,
bimeasurability of $\PsiJ$ implies that $\fJI$ has a unique, measurable
pushforward $\gJI:=\PsiI\circ\fJI\circ\PsiJ^{-1}$ on $\tspaceJ$. 
Hence, if the
conditional probabilities
$\PI[\,.\,|\ThetaI]$ defining the parametrized models
$\model_{\indI}$ are projective, we obtain 
$\fJI\PJ[\,.\,|\ThetaJ=\thetaJ]=\PI[\,.\,|\ThetaI=\gJI\thetaJ]$.
By applying a projective limit to all spaces, mappings and
conditionals indexed by $I$ in 
\eqref{diagram:parametric:model:projlim}, we obtain a projective
limit system of the form \eqref{diagram:parametric:model}
-- with each quantity indexed by $D$, respectively.
The resulting diagram again constitutes a parametrized model.
We refer to this model
as a \emph{projective limit model} in the following.

The definition immediately carries over to
Bayesian models: A projective
system of Bayesian models is defined by three projective systems
of standard Borel spaces $\famD{\xspaceI,\borel(\xspaceI),\fJI}$,
$\famD{\tspaceI,\borel(\tspaceI),\gJI}$ and
$\famD{\yspaceI,\borel(\yspaceI),\hJI}$, and by projective
families $\famD{\PXI[\XI|\ThetaI]}$ and $\famD{\PThetaI[\ThetaI|\YI]}$
of conditional distributions.
The uniqueness up to equivalence of the projective limit conditionals
(Theorem \ref{theorem:projlim:conditionals}) 
implies that the following diagram commutes:
\begin{diagram}[LaTeXeqno,small]
  \label{diagram:posterior:commutes}
  \PThetaD[\ThetaD|\YD] & \rTo^{\quad\XD^n=\xD^n\quad} & \PThetaD[\ThetaD|\XD^n,\YD]\\
  \dTo_{\gI} & & \dTo_{\gI}\\
  \PThetaI[\ThetaI|\YI] & \rTo^{\XI^n=\fI^n\xD^n} & \PThetaI[\ThetaI|\XD^n,\YI]
\end{diagram}
In other words, we obtain the same posterior regardless of whether we
(i) take projective limits of the finite-dimensional models
and then compute the infinite-dimensional posterior, or (ii) 
compute all finite-dimensional posteriors under marginal observations
and take the projective limit.

\subsection{Application of Pullbacks}
\label{sec:bayesian:pullback}

Pullbacks can be applied to parametrized and
Bayesian models in a manner largely analogous to projective limits. 
However, the pullback in general results in a restriction of the
abstract probability space:
If a probability measure $P=X(\abstmeasure)$
is pulled back under an injective map
$\phi:\txspace\rightarrow\xspace$, the resulting random
variable $\tX=\phi^{-1}\circ X$ is only defined on
$\tOmega:=X^{-1}\phi(\txspace)$. As a subset of the abstract
probability space, $\tOmega$ can always be assumed measurable,
and we will for simplicity assume that it is not a null set.
The corresponding restriction $\tabstmeasure$ 
of the abstract probability measure $\abstmeasure$ is the
conditional
$\tabstmeasure(\,.\,)=\abstmeasure[\,.\,|\tOmega]$, \ie
the abstract probability space underlying the pullback
measure $\tP=\tX(\tabstmeasure)$ is
$(\tOmega,\abstfield\cap\tOmega,\tabstmeasure)$.

Consider the  parametrized model $P[X|\Theta]$
described by \eqref{diagram:parametric:model}.
In this case, the entire diagram 
\eqref{diagram:parametric:model} may be pulled back to obtain
\begin{diagram}[LaTeXeqno,small]
  \label{diagram:pullback:model}
  \tOmega &\rTo^{\quad \tX^{\infty}\quad} & \txspace^{\infty} & \rTo^{\quad\tF\quad} &
  \pMeas(\txspace)\supset\tmodel &\rTo^{\quad\tPsi\quad} & \ttspace\\
  \dTo_{\inclusion_{\abstspace}}
  & & \dTo_{\phi}
  & & 
  & & \dTo_{\inclusion_{\tspace}}\\
  \abstspace &\rTo^{X^{\infty}} & \xspace^{\infty} & \rTo^{F} &
  \pMeas(\xspace)\supset\model &\rTo^{\Psi} & \tspace \\
\end{diagram}
The pullback is applicable only
for those values $\Theta=\theta$ with
${P^{\ast}[\txspace|\Theta=\theta]} =1$. Let $\ttspace\subset\tspace$
be the set of such values.
Denote the corresponding set of pullbacks 
$\tmodel$. The restriction $\tOmega$ induced
by the pullback is represented in the diagram
by the canonical inclusion mapping $\inclusion_{\abstspace}$.
The mappings $\tX$, $\tF$ and $\tPsi$ are
the restrictions of the mappings $X$, $F$ and $\Psi$ to the
respective restricted domains. 
Whenever $\theta\in\ttspace$,
we write $\tP[\,.\,|\tTheta=\theta]$ for the pullback measure of
$P[\,.\,|\Theta=\theta]$. This notation as a model parametrized
by $\tTheta$ is justified by the following lemma.
\begin{lemma}
  \label{lemma:pullback:parameterized:models}
  Let $\tnu_{\theta}(\,.\,)$ be the pullback of $P[X|\Theta=\theta]$.
  Then the function
  $(\tilde{A},\omega)\mapsto\tnu_{\Theta(\omega)}(\tilde{A})$
  is a regular conditional probability of $\tX$ given the
  random variable
  $\tTheta:=\tPsi\circ\tF\circ\tX^{\infty}$,
  \ie a regular version of 
  $\tabstmeasure[\tX\in\tilde{A}|\sigma(\tTheta)](\omega)$.
\end{lemma}

\begin{proof}
  The function $\tnu_{\Theta(\,.\,)}$ is the
  restriction
  $\tnu_{\Theta(\,.\,)}(A\cap\txspace)=P[A|\Theta]\vert_{\tOmega}$
  of the integrable, $\sigma(\Theta)$-measurable function
  $P[A|\Theta](\,.\,)$.
  Since $\sigma(\tTheta)=\sigma(\Theta)\cap\tOmega$, the mapping
  $\omega\mapsto\tnu_{\Theta(\omega)}(A\cap\txspace)$ is
  $\sigma(\tTheta)$-measurable for every $A\in\borelx$.
  The pullback of an integrable
  function preserves the integral (cf. \eqref{eq:pullback:preserves:integral}).
  Hence, $\tnu_{\Theta(\,.\,)}$ is a conditional given $\sigma(\Theta)$: For any $C\in\sigma(\Theta)$,
  \begin{equation*}
    \begin{split}
    \int_{C\cap\tOmega}\tnu_{\Theta(\omega)}(A\cap\txspace)d\tabstmeasure(\omega)
    =&
    \int_{C}\abstmeasure[X^{-1}A|\Theta](\omega)d\abstmeasure(\omega)\\
    =&\; \abstmeasure(A\cap C)
    =\tabstmeasure((X^{-1}A\cap\tOmega)\cap(C\cap\tOmega))\;.
    \end{split}
  \end{equation*}
\end{proof}

For a Bayesian model, the pullback is
consecutively applied to the sampling model
and to the prior. The pullback of $P[X|\Theta]$ 
induces a restriction of the parameter space from 
$\tspace$ to $\ttspace$. Lemma 
\ref{lemma:pullback:parameterized:models} guarantees
that the induced random variable $\tTheta$ is indeed
the parameter variable of the resulting model.
The prior
family $\PTheta[\Theta|Y]$ can hence be pulled back 
under $\inclusionT$. The pullback exists for
all $y\in\yspace$ 
with outer measure $P^{\theta,\ast}[\ttspace|Y=y]=1$, which 
in turn, by another application of Lemma
\ref{lemma:pullback:parameterized:models},
induces a restriction of the hyperparameter
space to $\tyspace\subset\yspace$.

\subsection{Nonparametric Evaluation}
\label{sec:bayesian:nonparametric}

The term \emph{nonparametric Bayesian model} usually implies that a
set of finite-dimensional measurements are explained by a posterior
distribution on an infinite-dimensional parameter space $\ttspace$.
In some models, the sampling distribution $\tP[\tX|\tTheta]$
is chosen to generate
finite-dimensional values given an instance 
$\tilde{\theta}$ of the infinite-dimensional parameter variable $\tTheta$;
the Dirichlet process construction in Sec.~\ref{sec:examples:DP} is an
example of such a model, where \eg $\txspace=\mathbb{R}$ and $\tilde{\theta}$ 
is a probability measure on $\mathbb{R}$.
For other models, such as the Gaussian process, it may be more convenient
to assume that finite-dimensional measurements are \emph{censored}
observations of infinite-dimensional random quantities. Which
of these assumptions is appropriate, and how a posterior is to
be computed under censored observations, depends on the model in
question. 

The setting can in general be formalized as follows.
Let $I_1,\dots,I_n\in D$ be index sets, and suppose measurements
$x_{\indI_{\ind{j}}}\in\xspace_{\indI_{\ind{j}}}$ are reported for $j=1,\dots,n$.
The nonparametric Bayesian model explains these measurements as being
generated by (i) drawing $\tilde{\theta}$ from the prior distribution;
(ii) generating $n$ samples $\tX^{(1)},\dots,\tX^{(n)}$ from 
$\tP[\tX|\tTheta=\tilde{\theta}]$; and finally, (iii) censoring the samples
as $x_{\indI_{\ind{j}}}=\phi_{\ind{I}_{\ind{j}}}\tX^{(j)}$. 
Whether the index sets $I_j$ are fixed or generated at random does
not affect the formalism, provided that their choice is stochastically
independent of the random variables in the model.
Since the censored observations $x_{\indI_{\ind{j}}}$ are represented as
projections of \emph{separate} instances $\tX^{(1)},\tX^{(2)},\dots$,
they are conditionally independent given $\tilde{\theta}$.
Asymptotically, we recover either
$\tilde{\theta}$, or a censored version of $\tilde{\theta}$, depending
on the index sets $I_j$ at which sample information is obtained.


\pdfoutput=1

\section{Sufficient Statistics}
\label{sec:sufficiency}

The purpose of this section is to show that the application of
sufficient statistics commutes with the application of projective
limits and pullbacks. If each element of a projective family
of parametrized models admits a sufficient statistic, the
projective limit of these functions is a sufficient statistic
for the projective limit model. Similarly,
the pullback of the sufficient statistic
is a sufficient statistic for the pullback model.

\begin{definition}[Sufficient statistic \citep{Halmos:Savage:1949}]
  Let $\P[X|\Theta]$ be a regular conditional probability. A
  $\sigma$-algebra $\Sfield\subset\abstfield$ is called {\em sufficient}
  for $\P[X|\Theta]$ if there is a probability kernel
  $k:\abstfield\times\abstspace\rightarrow[0,1]$, such that 
  (i) $\omega\mapsto k(B,\omega)$ is $\Sfield$-measurable for all
  $B\in\borelx$, and (ii) for all $B\in\borel(\xspace)$,
  \begin{align}
    \label{eq:def:sufficiency:abstspace}
    &
    \P[B|\Theta,\Sfield](\omega)=k(X^{-1}B,\omega) 
    &
    \abstmeasure\mbox{-a.s.}
  \end{align}
  If $\Sfield$ is sufficient and $(\Sspace,\borelS)$ is a measurable Polish space, then a 
  measurable mapping $S:\xspace\rightarrow\Sspace$ is called a 
  {\em sufficient statistic} for $P[X|\Theta]$ if $\Sfield=\sigma(S\circ X)$.
\end{definition}



\begin{theorem}[Sufficient $\sigma$-algebras and projective limits]
  \label{theorem:proj:lim:sufficient:field}
  Consider a projective limit model $\PD[\XD|\ThetaD]=\plim\famD{\PI[\XI|\ThetaI]}$.
  For each $I\in D$, let 
  $\SfieldI\subset\abstfield$ be a sufficient 
  $\sigma$-algebra for $\PI[\XI|\ThetaI]$.
  \begin{enumerate}
  \item If $\famD{\SfieldI}$ is directed, 
    $\SfieldD:=\sigma(\SfieldI; I\in D)$ is 
    sufficient for $\PD[\XD|\ThetaD]$.
  \item Let $\famD{\SspaceI,\borelSI,\hJI}$ be a projective system of measurable
    spaces and $\SI:\xspaceI\rightarrow\SspaceI$ projective measurable
    mappings. If each $\SI$ is sufficient for $\PI[\XI|\ThetaI]$,
    a sufficient statistic for $\PD[\XD|\ThetaD]$ is given by $\SD:=\plim\famD{\SI}$.
  \item Conversely,
  if any $\field\subset\abstfield$ is sufficient
  for a projective limit model $\PD[\XD|\ThetaD]$, it is sufficient
  for all marginals $\PI[\XI|\ThetaI]$.
  \end{enumerate}
\end{theorem}

\begin{proof}
  (1) We have to show that $\SfieldD$ satisfies \eqref{eq:def:sufficiency:abstspace},
  which is equivalent to $\XI\condind_{\SfieldI}\ThetaI$ 
  \citep[][Proposition 6.6]{Kallenberg:2001}.
  We will draw on two properties of conditional independence:
  Consider two $\sigma$-algebras $\mathcal{F}$ and $\mathcal{G}$.
  Firstly, if $D$ is any countable set, and $\lbrace\fieldI\rbrace_{\indI\in\indD}$
  a family of $\sigma$-algebras, then
  \begin{equation}
    \label{eq:condind:fact1}
    \mathcal{F}\condind_{\mathcal{G}}\fieldI \text{ for all } I
    \quad
    \Rightarrow
    \quad
    \mathcal{F}\condind_{\mathcal{G}}\sigma(\fieldI ; I\in D)\;.
  \end{equation}
  Since $\mathcal{G}$ is fixed, \eqref{eq:condind:fact1} is a direct consequence
  of the analogous result for unconditional independence \cite[][Corollary 2.7]{Kallenberg:2001}.
  Secondly, suppose the index set $D$ is a directed set and $\lbrace\fieldI\rbrace_{\indI\in\indD}$ a
  directed family. Then for any fixed $I_0\in D$, the following holds:
  \begin{equation}
    \label{eq:condind:fact2}
    \bigl(\;\mathcal{F}\condind_{\mathcal{G}}\field_{\ind{I}_0}
    \quad\text{ for all } J\succeq I_0\;\bigr) 
    \quad
    \Rightarrow
    \quad
    \bigl(\;
    \mathcal{F}\condind_{\mathcal{G}}\fieldJ\quad\text{ for all }J\in D
    \;\bigr)
  \end{equation}
  Under the assumptions of the theorem, $\SfieldI$ satisfies
  \eqref{eq:def:sufficiency:abstspace} for each $I$, or
  equivalently, $\XI\condind_{\SfieldI}\ThetaI$.
  Since the family of conditionals is projective, 
  we additionally know $\XI\condind_{\ThetaI}\ThetaJ$
  (Lemma \ref{lemma:cond:proj:criterion:1}).
  Therefore, by the chain rule, $\XI\condind_{\SfieldI}\ThetaJ$
  whenever $I\po J$, 
  and hence for all $J\in D$ according to \eqref{eq:condind:fact2}.
  Since $\sigma(\ThetaD)=\sigma(\ThetaI ; I\in D)$, we can apply
  \eqref{eq:condind:fact1} to obtain $\XI\condind_{\SfieldI}\ThetaD$.
  The fact that $\SfieldI\subset\SfieldD$ implies 
  $\XI\condind_{\SfieldD}\ThetaD$.
  The $\sigma$-algebra generated by the projective limit variable
  $\XD$ is simply $\sigma(\XI;I\in D)$, so another application of
  \eqref{eq:condind:fact1} yields
  $\XD\condind_{\SfieldD}\ThetaD$, which makes $\SfieldD$ sufficient
  for the projective limit model.

  \noindent (2) Since the functions $\SI$ are projective, the 
  family $\famD{\sigma(\SI)}$ of $\sigma$-algebras is directed,
  and $\sigma(\SD)=\sigma(\SI ; I\in D)$. The claim follows from 
  part (1) above.

  \noindent (3) Let $\kD$ be the kernel for which $\field$ and $\PD[\XD|\ThetaD]$
  satisfy \eqref{eq:def:sufficiency:abstspace}. We need to derive
  a suitable kernel $\kI$ for each $I\in D$.
  By Lemma \ref{lemma:cond:proj:criterion:1}, the marginals satisfy
  $\XI\condind_{\ThetaI}\ThetaJ$, and hence
  $\XI\condind_{(\ThetaI,\field)}\ThetaJ$. Since trivially
  also $\XI\condind_{(\ThetaI,\field)}\field$, the
  chain rule yields
  $\XI\condind_{(\ThetaI,\field)}(\ThetaJ,\field)$. Again by
  Lemma \ref{lemma:cond:proj:criterion:1}, the latter implies
  \begin{equation}
    \label{eq:lemma:sufficient:projection}
    (\fI\PD)[\XD|\ThetaD,\field]\eqae\PI[\XI|\ThetaI,\field] \;.
  \end{equation}
  Since $\field$ is sufficient for $\PD[\XD|\ThetaD]$, the 
  left-hand side of \eqref{eq:lemma:sufficient:projection}
  is equal to a kernel $\fI\kD$.
  Hence, $\kI:=\fI\kD$ is a $\field$-measurable kernel satisfying
  $\kI(\AI,\omega)\eqae\PI[\AI|\ThetaI,\field](\omega)$,
  which makes $\field$ sufficient for $\PI[\XI|\ThetaI]$.
\end{proof}

To make the construction of the sufficient statistics of a parametrized
stochastic process fully compatible with the construction of the
process itself requires an analogous result for pullbacks.

\begin{proposition}[Sufficiency and pullbacks]
  \label{theorem:suffstat:pullback}
  Let $\tP[\tX|\tTheta]$ be the pullback of a parametrized model
  $P[X|\Theta]$ under a Borel embedding
  $\inclusionX:\txspace\rightarrow\xspace$.
  If $S:\xspace\rightarrow\Sspace$ is a sufficient statistic for
  $P[X|\Theta]$, then $\tS:=S\circ\inclusionX$ is sufficient for
  $\tP[\tX|\tTheta]$. 
\end{proposition}

\begin{proof}
  We have to show that $\tP[\tX|\tTheta]$ and $\tS$ satisfy
  \eqref{eq:def:sufficiency:abstspace} 
  for a suitable kernel $\tk$, which we define as follows.
  Let $k(A,\omega)$ be the kernel 
  \eqref{eq:def:sufficiency:abstspace} for $S$ and $\P[X|\Theta]$.
  The domain of $\tS\circ\tX$ is $(\tS\circ\tX)^{-1}\Sspace=\tOmega$.
  For any $A\in\sigma(X)$, define $\tk$ as
  $\tk(A\cap\tOmega,\,.\,):=k(A,\,.\,)\vert_{\tOmega}$.
  By definition, $\omega\mapsto\tk(A\cap\tOmega,\omega)$
  is measurable with respect to 
  $\sigma(S\circ X)\cap\tOmega=\sigma(\tS\circ\tX)$,
  which implies measurability with respect to the finer
  $\sigma$-algebra
  $\sigma(\tTheta,\tX\circ\tS)=\sigma(\Theta,X\circ S)\cap\tOmega$.
  Hence, 
  \eqref{eq:def:sufficiency:abstspace} is satisfied if the integral
  of $\tk$ matches that of the conditional probability for each set
  in $\sigma(\tTheta,\tX\circ\tS)$.  Let $\tC=C\cap\tOmega$ be any such set.
  As pullbacks preserve integrals
  in the sense of \eqref{eq:pullback:preserves:integral},
  \begin{equation*}
    \begin{split}
      \int_{\tC}\tk(\tA,\tS\circ\tX(\omega))d\tabstmeasure(\omega)
      =&
      \int_{C}k(A,S\circ X(\omega))d\abstmeasure(\omega)
      =
      \abstmeasure(X^{-1}A\cap C)\\
      =&
      \ \tabstmeasure(X^{-1}A\cap C\cap\tOmega)
      =
      \int_{\tC}\tP[\tA|\tTheta,\tS](\omega)d\tabstmeasure(\omega)\;.
    \end{split}
  \end{equation*}
  Thus $\tS$, $\tk$ and $\tP[\tX|\tTheta]$ satisfy
  \eqref{eq:def:sufficiency:abstspace}, and
  $\tS$ is sufficient for $\tP[\tX|\tTheta]$.  
\end{proof}

We conclude this section with a result on minimality, 
\ie the question whether a ``smallest'' sufficient $\sigma$-algebra
exists for a given model. The concept is closely related to
that of a minimal sufficient statistic
-- a sufficient statistic to which any statistic sufficient for the
model can be reduced by transformation -- but the two are
not equivalent \citep{Landers:Rogge:1972}.

\begin{definition}[Minimal sufficient $\sigma$-algebra \citep{Landers:Rogge:1972}]
  A $\sigma$-algebra $\Sfield_0\subset\abstfield$ is called
  \emph{minimal sufficient} for $P[X|\Theta]$ if it is sufficient,
  and if every other sufficient $\sigma$-algebra $\field$ satisfies:
    \begin{equation}
      \label{eq:def:minimality}
      \forall A\in\Sfield_0\;\exists C\in\field:\qquad P[A\triangle
      C|\Theta=\theta]=0\qquad\text{ for all } \theta\in\tspace\;.
    \end{equation}
\end{definition}

Intuitively, minimality captures the idea that any $\sigma$-algebra 
$\field$ can only be sufficient for the model if it contains all
information contained in $\Sfield_0$ (though this interpretation
is inaccurate in the undominated case, as pointed out by
\citet{Burkholder:1961}). However, instead of demanding
$\Sfield_0\subset\field$, and hence that every set in $\Sfield_0$
is also in $\field$, we only require that each set in $\Sfield_0$
be indistinguishable from a set in $\field$ under the resolution
of the model.

A minimal sufficient $\sigma$-algebra always exists
if the model $P[X|\Theta]$ in question is dominated. In undominated
models, a sufficient $\sigma$-algebra can -- rather contrary to 
intuition -- be contained in a \emph{finer} $\sigma$-algebra which is
\emph{not} sufficient, and a minimal sufficient $\sigma$-algebra 
need not exist \citep{Burkholder:1961}.
However, as the following theorem shows, existence is 
guaranteed if the model is constructed as a
projective limit from dominated marginals.
This implies, for example, that the
Dirichlet process on the line admits a minimal sufficient $\sigma$-algebra,
even though it is undominated.

\begin{proposition}[Minimal sufficiency]
  \label{theorem:suffstat:minimal}
  Suppose that each $\sigma$-algebra $\SfieldI$ as specified in
  Theorem \ref{theorem:proj:lim:sufficient:field} is minimal
  sufficient for $\PI[\XI|\ThetaI]$. Then 
  $\SfieldD=\plim\famD{\SfieldI}$ is minimal
  sufficient for $\PD[\XD|\ThetaD]=\plim\famD{\PI[\XI|\ThetaI]}$.
\end{proposition}

\begin{proof}
  $\SfieldD$ is sufficient by Theorem
  \ref{theorem:proj:lim:sufficient:field};
  we have to verify \eqref{eq:def:minimality}.
  Let $\field\subset\abstfield$ be any sufficient $\sigma$-algebra
  for $\PD[\XD|\ThetaD]$. By Theorem \ref{theorem:proj:lim:sufficient:field},
  $\field$ is sufficient for all $\PI[\XI|\ThetaI]$, which implies
  that \eqref{eq:def:minimality} is satisfied if
  $A\in\cup_{\indI}\SfieldI$.
  For the general case $A\in\SfieldD$, observe that
  the set system $\cup_{\indI}\SfieldI$ is both an algebra and a generator
  of $\SfieldD$. By the basic theorem on approximation of a measure
  on a subalgebra \citep[][Theorem 5.7]{Bauer:1996}), any set $A\in\SfieldD$
  can hence be approximated by a sequence of sets
  $A_n\in\cup_{\indI}\SfieldI$ 
  such that $\lim_n\abstmeasure[A_n\triangle
  A|\Theta=\theta]=0$.
  Since each $A_n$ satisfies \eqref{eq:def:minimality}, there
  is a corresponding set $C_n\in\field$ such that
  $\abstmeasure[A_n\triangle C_n|\Theta=\theta]=0$.
  Then $\lim_n\abstmeasure[A \triangle C_n|\Theta=\theta]=0$,
  and therefore $\abstmeasure[A
  \triangle\cup_{n}C_n|\Theta=\theta]=0$. Since $\field$ is a
  $\sigma$-algebra, $\cup_nC_n\in\field$, and $A$ satisfies
  \eqref{eq:def:minimality} for $C:=\cup_nC_n$.
\end{proof}

\pdfoutput=1

\section{Conjugacy}
\label{sec:conjugacy}

The posterior of a Bayesian model is a regular conditional
probability, and always exists if the model is defined on Polish spaces.
However, since the abstract components of the model -- the probability space
$(\abstspace,\abstfield,\abstmeasure)$
and the random variables
$X$ and $\Theta$ -- are not given explicitly, there is in general
no way to deduce the posterior from the sampling distribution and the prior.
The problem is solved by Bayes' theorem whenever the sampling
distribution is dominated, \ie if $P[X|\Theta]$ has a conditional density
\citep[][Theorem 1.31]{Schervish:1995}. This need not be the case in the
infinite-dimensional setting of Bayesian nonparametrics.
For a certain class of Bayesian models, so-called \emph{conjugate} models,
the posterior can be specified without appealing to Bayes' theorem.
Virtually all nonparametric Bayesian models studied in the
literature are of this type (see \eg 
\citep{Walker:Damien:Laud:Smith:1999}).

\begin{definition}[Conjugate Bayesian model]
  Let $P[X|\Theta]$ and $\PTheta[\Theta|Y]$ specify a Bayesian model.
  Let $(\Tn)_n$ be a family of measurable mappings 
  $\Tn:\xspace^n\times\yspace\rightarrow\indexspace$ with values in a
  Polish space. The family is called a
  \emph{posterior index} of the model if there
  exists a probability kernel
  $\kernel:\borelT\times\indexspace\rightarrow[0,1]$
  such that
  \begin{equation}
    \label{eq:def:posterior:index}
    \PTheta[A|X^n=(x_1,\dots,x_n),Y=y]\eqae k(A,\Tn(x_1,\dots,x_n,y))
  \end{equation}
  for every $A\in\borelT$. A Bayesian model is called \emph{conjugate}
  if there is a posterior index for which $\indexspace\subset\yspace$
  and the associated kernel $\kernel$ satisfies
  \begin{equation}
    \label{eq:def:conjugacy}
    k(A,y')=\PTheta[A|Y=y'] \qquad\qquad\text{ for all }y'\in\yspace\;.
  \end{equation}
\end{definition}
Apparently, any Bayesian model admits the identity
$\Tn:=\mbox{Id}_{\xspace^n\times\yspace}$ as a trivial
posterior index. By \eqref{eq:def:conjugacy}, a conjugate
posterior is ``in the same family'' as the prior, a model 
property commonly
referred to as \emph{closure under sampling}
\cite{Raiffa:Schlaifer:1961}.

In a projective system, we have to consider a family of spaces
$\mathcal{W}_{\indI}$ as the respective ranges of the posterior
indices $(\TnI)_n$. As for the hyperparameter spaces $\yspaceI$, we
will denote the projectors on these spaces by $\hJI$, since
$\mathcal{W}_{\indI}$ and $\yspaceI$ are either subsets of one
another, or can without loss of generality be assumed to be
contained in a common superspace.
The following theorem states that the posterior updates of a nonparametric
Bayesian model have the same ``functional form'' as those of its finite-dimensional
marginals. It also implies that conjugacy of the model requires
conjugate marginals.
\def\TnJ{\Tn_{\indJ}}
\begin{theorem}[Conjugacy in projective limit models]
  \label{theorem:conjugate:projective:limits}
  Let $\famD{\PI[\XI|\ThetaI]}$ and $\famD{\PThetaI[\ThetaI|\YI]}$
  define a projective family of Bayesian models.
  \begin{enumerate}
  \item Let $(\TnI)_n$ be posterior indices and projective, \ie
    \begin{equation}
      \label{eq:posterior:index:projective}
      \TnI\circ(\hJI\otimes\fJI^n) = \hJI\circ\TnJ\qquad\text{ for }I\in D, n\in\mathbb{N}.
    \end{equation}
    Then the mappings $\TnD:=\plim\famD{\TnI}$ form
    a posterior index of the projective limit model.
    If each marginal model is conjugate under $(\TnI)_n$,
    the projective limit model is conjugate under $(\TnD)_n$.
  \item Conversely, let the projective limit be conjugate.
  If the canonical mappings $\fI$, $\hI$ are surjective,
  the marginals of the model
  are closed under sampling. If additionally each of the mappings
  $\fI$ and $\hI$ is open or closed, the marginals are conjugate
  and their posterior indices are pushforwards of 
  $(\TnD)_n$ satisfying
  \begin{equation}
    \label{eq:theorem:conjugate:marginals:posterior:index}
    \TnI\circ(\fI^n\otimes\hI) = \hI\circ\TnD \qquad\text{ for }
    I\in D, n\in\mathbb{N}\;.
  \end{equation}
  \end{enumerate}
\end{theorem}
\noindent (Proof: App.~\ref{sec:proof:conjugacy}.)

Consequently, a conjugate nonparametric
Bayesian model can only be obtained from marginals which are
closed under sampling. Dropping either of the two assumptions
in the theorem -- that the model is defined as a projective
limit and that the canonical mappings be surjective --
does not lift this restriction. 
If the
model is not explicitly assumed to be a projective limit,
a family of marginals can always be obtained by defining
$\ThetaI:=\gI\ThetaD$ and 
$\PI[\XI|\ThetaI]:=\mean{(\fI\PD)[\XD|\ThetaD]|\sigma(\ThetaI)}$,
etc. The components so obtained form projective families with
the initial model as their limit, and the theorem is applicable.
Similarly, if the canonical
mappings are not assumed surjective, we simply obtain a more
technical statement of the theorem which requires closure
under sampling on the images of the canonical mappings. 
The generalization obtained in this way is trivial, since
all measures used in the construction have to concentrate on
these images.
We also note, in the context of part (2), that the projectors $\projectorI$ in
a countable product of Polish spaces are always open mappings.

Similar to projective limits, pullbacks preserve conjugacy:
\begin{proposition}[Pullbacks of conjugate models]
  \label{theorem:conjugate:pullbacks}
  Let the Bayesian model specified by $P[X|\Theta]$ and
  $\PTheta[\Theta|Y]$ be conjugate, with posterior index $(\Tn)_n$.
  Let $\tP[\tX|\tTheta]$ and $\tPTheta[\tTheta|\tY]$ be the
  respective pullbacks under $\inclusionX$ and $\inclusionT$.
  Then the Bayesian model specified by
  $\tP[\tX|\tTheta]$ and $\tPTheta[\tTheta|\tY]$ is conjugate,
  with posterior index 
  given by the pullbacks of $(\Tn)_n$ as
  \begin{equation}
    \label{eq:posterior:index:pullback}
    \tTn:=\inclusionY^{-1}\circ\Tn\circ(\inclusionX^n\otimes\inclusionY)\;.
  \end{equation}
\end{proposition}

\begin{proof}
  As an arbitrary subset of a Polish space, $\ttspace$ is
  separable, but not in general Polish, and conditional
  probabilities are not guaranteed to be regular.
  Since the pullback is defined by restriction, which preserves
  measurability, $\tPTheta(\tTheta|\tY)$ nonetheless constitutes
  a well-defined regular conditional probability.
  Lemma \ref{lemma:pullback:parameterized:models}
  ensures that the spaces $\txspace$, $\ttspace$ and $\tyspace$
  all correspond to the same subset $\tOmega$ of the abstract
  probability space. Equation \eqref{eq:posterior:index:pullback}
  is an immediate consequence of the definitions of posterior
  indices and pullbacks of parametric models.
\end{proof}

As an example of the previous results, we consider one of the most
widely used Bayesian nonparametric models, a Gaussian process
model for regression under uniform measurement noise
\citep{Zhao:2000,Rasmussen:Williams:2006}. The 
purpose of the example is to provide concrete illustration of the abstract
quantities above, and we sacrifice rigor for brevity and refer
to Sec.~\ref{sec:examples} for more detailed constructions.

\def\tgamma{\tilde{\gamma}}
\def\tTn{\tilde{T}^{(n)}}
\def\L2{\mathcal{L}_2}
\def\ttheta{\tilde{\theta}}
\begin{example}[Gaussian process regression under white noise]
  \label{example:GP:regression}
  Consider a regression problem on $[0,1]$, in which measurements
  $x_s\in\mathbb{R}$ are recorded at covariate locations $s\in [0,1]$.
  Assume each measurement to be a value $\theta_s$ corrupted by standard
  white noise $\varepsilon_s\sim\mathcal{N}(0,1)$, that is, $x_s=\theta_s+\varepsilon_s$.
  Since the noise is discontinuous, the function space $\txspace:=\L2[0,1]$
  is a more adequate setting than the set of continuous functions considered
  in Example \ref{example:GP:C}.
  Let $(e_i)_{i\in\mathbb{N}}$ be an orthonormal basis of $\L2[0,1]$.
  Any $\tx\in\L2[0,1]$ is uniquely representable as
  $\tx = \sum_{i}x_{\indi}e_i$, where $x_{\indi}=\bigl<\tx,e_i\bigr>$.
  The mapping $\phi:\tx\mapsto (x_{\indi})_{i\in\mathbb{N}}$ is an isomorphism
  of the separable Hilbert spaces $\L2[0,1]$ and $\ell_2$.
  Since $\ell_2\subset\mathbb{R}^{\mathbb{N}}$, we 
  choose the product projective limit $\xspaceD=\mathbb{R}^{\mathbb{N}}$
  with canonical mappings $\fI:=\projectorI$.
  In the terminology of Definition \ref{def:cr:stoch:proc},
  the set $\tImage=\ell_2$ is the image of $\L2[0,1]$ under
  the Borel embedding $\phi$.

  To obtain a valid model on $\L2[0,1]$ as a pullback of the 
  Gaussian projective limits on $\xspaceD$ and $\tspaceD$,
  we need to know that the models assign outer measure 1
  to the subset $\tImage=\ell_2$ (cf. Sec.~\ref{sec:bayesian:pullback}).
  Gaussian processes with realizations in $\ell_2$ -- or, in
  our terminology, Gaussian projective limits which satisfy 
  $\PD^{\ast}(\ell_2)=1$ -- are characterized by a well-known
  result \citep[][Theorem 3.2]{Kuo:1975}: 
  Denote by $\mathcal{S}(\ell_2)$ the set of all positive
  definite Hermitian operators on $\ell_2$ 
  of ``trace class'', \ie with
  finite trace $\mbox{tr}(\Sigma)<\infty$. The Gaussian projective
  limit $\PD$ satisfies $\PD^{\ast}(\ell_2)=1$
  if and only if there are $m\in\ell_2$ and $\Sigma\in\mathcal{S}(\ell_2)$
  such that 
  \begin{equation}
    \mean{\XI}=\fI(m) \qquad\text{ and }\qquad 
    \mbox{Cov}[\XI]=(\fI\otimes\fI)(\Sigma) \;.
  \end{equation}

  To define a Bayesian model, we choose $\ttspace=\txspace$,
  $\tspaceD=\xspaceD$ and $\gI=\fI$. To define a projective family
  of priors, let 
  $m\in\ell_2$, $\Sigma\in\mathcal{S}_{\ell_2}$, and define the
  measures $\PThetaI(\ThetaI|\YI)$ as the Gaussian measures
  on $\tspaceI=\mathbb{R}^{\indI}$ with means $\gI(m)$ and
  covariance matrices $(\gI\otimes\gI)(\Sigma)$. 
  The hyperparameter spaces are therefore 
  $\yspaceI:=\mathbb{R}^{\indI}\times\mbox{Sym}(I,\mathbb{R})$, where 
  $\mbox{Sym}(I,\mathbb{R})$ is the symmetric cone of 
  real-valued $|I|\times|I|$ s.p.d.\ matrices. The projector $J\rightarrow I$
  on the latter deletes all rows and columns indexed by elements of 
  $J\setminus I$. For the white-noise 
  observation model, let $\mathbb{I}\in\mathcal{S}(\ell_2)$ 
  be the identity operator. Each marginal is chosen
  as the Gaussian conditional $\PI[\XI|\ThetaI;\mathbb{I}]$,
  \ie conditional on the random mean $\ThetaI$ for fixed unit covariance.
  The priors form a projective family of measures and the
  observation models, by Lemma \ref{lemma:cond:proj:criterion:2}, 
  a projective family of conditional distributions.
  A conjugate posterior index of the model is given by
  \begin{equation*}
    \label{eq:gaussian:parameter:updates}
    T^{(n)}_{\indI}:\;\;
    (\xI^n,m_{\indI},\Sigma_{\indI})
    \;\mapsto\;
    \bigl(m_{\indI}-\Sigma_{\indI}(\Sigma_{\indI}+\frac{1}{n}\mathbb{I}_{\indI})^{-1}(m_{\indI}-\frac{1}{n}\sum_{j=1}^n x_{\indI}),
    \Sigma_{\indI}(\Sigma_{\indI}+\frac{1}{n}\mathbb{I}_{\indI})^{-1}\bigr)\;.
  \end{equation*}
  Since the covariance of the observation model is the fixed identity matrix $\mathbb{I}_{\indI}$,
  it is not a hyperparameter, and hence formally part of the definition of
  the mapping rather than an argument.

  \def\tm{\tilde{m}}
  The posterior index $\tTn$ of the nonparametric model can be 
  constructed as follows.
  We define a candidate function which mimics the functional form
  of the finite-dimensional posterior indices $\TnI$:
  For $\tx\in\ell_2^n$, $\tm\in\ell_2$ and $\tilde{\Sigma}\in\mathcal{S}(\ell_2)$,
  define a mapping 
  $\ell_2^n\times\ell_2\times\mathcal{S}(\ell_2)\mapsto
  \ell_2\times\mathcal{S}(\ell_2)$ as
  \begin{equation*}
    \tilde{T}^{(n)}:\;\;
    (\tx^n,\tm,\tilde{\Sigma})
    \;\mapsto\;
    \bigl(\tm-\tilde{\Sigma}(\tilde{\Sigma}+\tilde{\mathbb{I}})^{-1}(\tm-\frac{1}{n}\sum_{j=1}^n\tx^{(n)}),
    \tilde{\Sigma}(\tilde{\Sigma}+\tilde{\mathbb{I}})^{-1}\bigr) \;.
  \end{equation*}
  It is straightforward to verify that (i) the maps $T^{(n)}_{\indI}$ are
  the finite-dimensional projections of $\tilde{T}^{(n)}$ under the 
  projectors $\fI\circ\phi$ and $\hI\circ\phi$, and that (ii)
  the family of maps $\famD{T^{(n)}_{\indI}}$ is projective (\ie satisfies
  \eqref{eq:def:projlim:mappings}).
  Therefore, $\tilde{T}^{(n)}$ indeed
  coincides with the unique projective limit map on the relevant subspace
  $\ell_2^n\times\ell_2\times\mathcal{S}_{\ell_2}$ of the projective limit
  space. 
  By Theorem \ref{theorem:conjugate:projective:limits}, the projective
  limit model is conjugate.
  Since $\tilde{T}^{(n)}$ maps $\ell_2^n\times\ell_2\times\mathcal{S}_{\ell_2}$
  into $\ell_2\times\mathcal{S}_{\ell_2}$,
  the posterior again assigns full outer measure to $\ell_2=\phi(\L2[0,1])$.
  By Theorem \ref{theorem:conjugate:pullbacks}, the pullback of the model
  under $\phi$ is a conjugate Gaussian process model on $\ell_2$.
  The sufficient statistics $\SI=\mbox{Id}_{\mathbb{R}^{\indI}}$ of the
  marginals are trivially projective, and by Theorem
  \ref{theorem:proj:lim:sufficient:field} and Proposition \ref{theorem:suffstat:pullback},
  the pullback $\tS=\mbox{Id}_{\ell_2}$ of their projective limit
  is a sufficient statistic of the Gaussian process model. 
\end{example}

Since conjugacy in parametric models
is, with few exceptions, a property of exponential family models,
we can interpret most conjugate nonparametric Bayesian models
as infinite-dimensional analogues of the exponential family. 
Conjugate priors of exponential family models
are characterized by a linear arithmetic in parameter
space, as shown by \citet{Diaconis:Ylvisaker:1979}. In particular,
suppose that the marginals $\PI[\XI|\ThetaI]$ and 
$\PThetaI[\ThetaI|\YI]$ are exponential family marginals with
canonical conjugate priors. With respect to suitable carrier measures
on the spaces $\xspaceI$ and $\tspaceI$, the marginals are then
defined by conditional densities
\begin{equation}
  \label{eq:expfam:densities}
  p^x_{\indI}(\xI|\thetaI)=\frac{H_{\indI}(\xI)}{Z_{\indI}(\thetaI)}e^{\left<\SI(\xI),\thetaI\right>}
  \qquad
  p^{\theta}_{\indI}(\thetaI|\lambda,\yI)
  =
  \frac{e^{\left<\thetaI,\gammaI\right>-\lambda\log
      Z_{\indI}(\thetaI)}}{K_{\indI}(\lambda,\gammaI)} \;.
\end{equation}
The function $\SI:\xspaceI\rightarrow\SspaceI$ is a sufficient
statistic. Its range is a Polish topological vector space
$\SspaceI$, which contains the parameter space $\tspaceI$ as
a subspace, and is equipped with an inner product
$\left<\,.\,,\,.\,\right>$. $H_{\indI}$ denotes a
non-negative function, and $Z_{\indI}$, $K_{\indI}$ are 
normalization functions.
The prior is parametrized by $\lambda\in\mathbb{R}_+$,
which determines concentration, and $\gammaI\in\chull(\SI(\xspaceI))$,
the convex hull of the image of $\SI$.

In our previous notation, the Bayesian model defined by 
\eqref{eq:expfam:densities} has hyperparameter space 
$\yspaceI:=\mathbb{R}_+\times\SspaceI$, with $\yI=(\lambda,\gammaI)$.
The posterior under observations $\xI^n=(\xI^{(1)},\dots,\xI^{(n)})$ 
is given by the density $p_{\indI}(\thetaI|\TnI(\xI^n,\yI))$, where
the posterior index updates hyperparameters as 
$\lambda\mapsto\lambda+n$ and 
$\gammaI\mapsto\gammaI+\sum_k\SI(\xI^{(k)})$.
Application of Theorems
\ref{theorem:conjugate:projective:limits} 
and \ref{theorem:conjugate:pullbacks}
results in an analogous representation for nonparametric models:

\begin{corollary}[Exponential family marginals]
  \label{corollary:expfam}
  Let $\famD{\PI[\XI|\ThetaI]}$ be a projective family of
  exponential family models with sufficient statistics $\SI$,
  and let $\famD{\PThetaI[\ThetaI|\YI]}$ be the family of 
  corresponding canonical conjugate priors. If the 
  priors and the sufficient
  statistics both form projective families, the projective 
  limit Bayesian model is conjugate with posterior index
  \begin{equation}
    \label{eq:exp:fam:posterior:index}
    \TnD(\xD^n,\yD):=\bigl(\lambda+n,\gammaD+\sum_k\SD(\xD^{(k)})\bigr) \;,
  \end{equation}
  where $\yD=(\lambda,\gammaD)$ and $\SD:=\plim\famD{\SI}$ is the sufficient statistic
  of $\PD[\XD|\ThetaD]$.
  Analogously, if the model is pulled back under $\phi:\txspace\rightarrow\xspaceD$ as in
  \eqref{diagram:pullback:model},
  \begin{equation}
    \label{eq:exp:fam:posterior:index:pullback}
    \tTn(\tx^n,\ty):=\bigl(\lambda+n,\tgamma+\sum_k\tS(\tx^{(k)})\bigr) \;,
  \end{equation}
  is a conjugate posterior index of the pullback model.
\end{corollary}
An example of such a posterior is the Dirichlet process
on the line with concentration $\alpha$ and base measure $G_0$,
for which posterior parameters are updated under observations
$v^{(1)},\dots,v^{(n)}$ as
\begin{equation}
  (v^{(1)},\dots,v^{(n)},\alpha\cdot G_0) \;\mapsto\; \frac{n}{n+\alpha}\sum_{k=1}^n \delta_{v^{(k)}} + \frac{\alpha}{n+\alpha} G_0 \;.
\end{equation}
The next section covers this example in detail.
The Gaussian process regression above
is an instance
of Corollary \ref{corollary:expfam} as well, although
our formulation in Example \ref{example:GP:regression}
uses the standard parametrization of the Gaussian, rather than
an exponential family parametrization adapted
to \eqref{eq:exp:fam:posterior:index:pullback}.

\pdfoutput=1

\section{Examples}
\label{sec:examples}

\def\Part{\Pi}
\def\PartI{\Part_{\indI}}
\def\PartJ{\Part_{\indJ}}
\def\mI{m_{\indI}}
\def\mJ{m_{\indJ}}
\def\bgammaI{\bar{\gamma}_{\indI}}
\def\piD{\pi_{\indD}}

Two detailed construction examples are given 
in this section to illustrate
our results: The well-known Dirichlet process
\cite{Ferguson:1973,Kingman:1975}, and a new
nonparametric Bayesian model on the infinite
symmetric group. The steps of both constructions are 
(i) the definition of
projective systems to obtain $\xspaceD$ and
$\tspaceD$, (ii) the definition of finite-dimensional
priors and likelihoods for each $I\in D$
to define a projective limit Bayesian model,
and (iii) a pullback step to ensure that the models
concentrate on the desired subspace of interest
-- the set of probability measures and the infinite
symmetric group, respectively.
By means of the results in Secs.~\ref{sec:sufficiency} and
\ref{sec:conjugacy}, sufficiency and conjugacy properties
of the models can then be read off from the properties
of the marginals.

\subsection{Dirichlet Process Priors}
\label{sec:examples:DP}

In this example, $\tPTheta$ is a Dirichlet process and
$\tP$ its conjugate observation model.
The domain of the Dirichlet process is
assumed to be a Polish measurable space $(V,\borelV)$,
\ie random measures
drawn from the process are convex combinations of the form 
$\thetaD=\sum_{k\in\mathbb{N}}c_k\delta_{v_k}$ with $v_k\in V$.

\subsubsection{Projective System}
The finite-dimensional marginals will be Dirichlet and multinomial
distributions. \citet{Ferguson:1973} noted that a particularly
intuitive way to index such distributions is to choose
each $I\in D$ as a finite, measurable partition $I=(A_1,\dots,A_{|I|})$
of $V$. The $|I|$-dimensional Dirichlet distribution
$\PThetaI$ can then be interpreted as a random measure
on the finite $\sigma$-algebra $\sigma(I)$ generated by
the sets in $I$. Let $\parts(\borelV)$ be the set of
all finite partitions $I=(A_1,\dots,A_{|I|})$ with $A_i\in\borelV$.
This set is itself not an adequate choice for $D$, since it is
uncountable unless $V$ is finite. However, since $V$
is Polish, there exists a countable algebra $\mathcal{Q}\subset\borelV$
which generates $\borelV$.
Any probability measure on $\borelV$ can,
by Carath\'{e}odory's extension theorem,
be unambiguously represented by its restriction to
$\mathcal{Q}$. Bearing this in mind, we define 
$D:=\partsQ$ as the set of finite partitions with $A_i\in\mathcal{Q}$.
A partial order on $D$ is defined
by $I\po J$ if and only if $I\cap J=J$, that is, if $J$ is a
refinement of the partition $I$.

For each index $I=(A_1,\dots,A_{\mI})$ in $D$, the marginal spaces are
chosen as the spaces corresponding to a Dirichlet-multinomial
Bayesian model over $\mI$ categories:
Let the parameter space $\tspaceI$ be the set of probability
distributions on the $\sigma$-algebra generated by $I$, \ie
the unit simplex 
$\simpI\subset\mathbb{R}^{\indI}$. The hyperparameter
space of a Dirichlet model on $\simpI$ is 
$\yspaceI:=\mathbb{R}_{>0}\times\simpI$.

To define the observation spaces $\xspaceI$, we interpret the
sets $A_i\in I$ as categories or ``bins'' of a multinomial distribution.
A sample in category $A_i$ can be encoded as $\lbrace\XI=A_i\rbrace$.
Hence, $\XI$ takes values in $\xspaceI:=I$. Both the topology
$\topI$ and Borel sets $\borelI$ on $\xspaceI$ are generated by 
the singleton events $\lbrace A_1\rbrace,\dots,\lbrace A_{m_{\indI}}\rbrace$.

To define suitable projectors, consider a pair $I\po J$
of indices, where $I=(A_1,\dots,A_{\mI})$ and $J=(A'_1,\dots,A'_{\mJ})$.
Any set $A_i\in I$ is the union of some
sets in $J$, hence $A_i=\cup_{j\in J_i}A'_j$ for some $J_i\subset [\mJ]$.
Let $\thetaJ\in\simpJ$ be a finite probability distribution
and $A'_j\in J$. We define
\begin{equation}
  \label{eq:example:DP:def:f:g}
  \fJI(A'_j) = A_i 
  \quad\text{ for $j\in J_i$, \quad and }\quad
  (\gJI\thetaJ)_i:=\sum_{j\in J_i}(\thetaJ)_j\;.
\end{equation}
In words, for any coarsening of a finite set of events $J$ 
to $I$, $\fJI$ maps $A'_j$ to the coarser event containing it,
and $\gJI$ sums the corresponding 
probabilities. Since the model is of
conjugate exponential family type, 
the projections $\hJI:\yspaceJ\rightarrow
\yspaceI$ of hyperparameters are given by
$\hJI:=\mbox{Id}_{\mathbb{R}_+}\otimes\gJI$.
The families of mappings $\fJI:J\rightarrow I$
and $\gJI:\simpJ\rightarrow\simpI$ are
continuous, surjective and satisfy \eqref{eq:def:canonical:map}.
It is straightforward to verify that 
$\famD{\xspaceI,\borel(\xspaceI),\fJI}$ and
$\famD{\tspaceI,\borel(\tspaceI),\gJI}$
are projective systems of measurable Polish spaces.
Properties of the hyperparameter spaces $\yspaceI$
carry over immediately from $\tspaceI$.

What are the projective limit spaces 
$\xspaceD$ and $\tspaceD$ defined by the two systems? 
The set $\xspaceD$ consists of
all collections of the form 
\begin{equation}
  \xD=\lbrace C_{\indI}\in I | I\in D, C_{\indI}\supset C_{\indJ} \text{ whenever } I\po J\rbrace \;.
\end{equation}
Whereas a draw from $\XI$ selects a single random set $C_{\indI}\in I$, a draw from
$\XD$ selects one random set $C_{\indI}$ for each $I$. A single, ``smallest'' set can
be associated with each $\xD\in\xspaceD$ by defining $\lim\xD:=\cap_{I}C_{\indI}$.
Unlike the constituent sets $C_{\indI}$, the set $\lim\xD$ is not in general an
element of $\mathcal{Q}$, and we have 
$\mathcal{Q}\subset\lbrace\lim\xD|\xD\in\xspaceD\rbrace\subset\borelV$.
In particular, the proof of Lemma 
\ref{lemma:examples:DP:concentration:likelihood} below shows that 
the set $\lbrace\lim\xD|\xD\in\xspaceD\rbrace$ contains all singleton
$\lbrace v \rbrace$ for points $v\in V$, which are not contained in 
the countable set $\mathcal{Q}$. In analogy to the interpretation
of $\XI$ as an event $A_i\in I$, we can interpret $\XD=\xD$ as the
event $\lim\xD$ in $\borelV$.

The projective limit $\tspaceD$ of parameter spaces is the
set of all charges, \ie of \emph{finitely} additive probabilities on the algebra
$\mathcal{Q}$. The space $(\tspaceD,\borel(\tspaceD))$ contains
the set $\pMeas(\mathcal{Q})$ of countably additive probability
measures as a measurable subset \citep[][Proposition 9]{Orbanz:2010}.
For any set function $G\in\tspaceD$, the canonical maps 
$\gI:\tspaceD\rightarrow\simpI$ are the evaluations 
$G\mapsto(G(A_1),\dots,G(A_{\mI})$.
The fact that the space $\pMeas(\mathcal{Q})$ cannot directly
be defined as a projective limit of finite-dimensional simplices
is an example for the projective limit's ability to encode
a finitary property (finite additivity), but not an infinitary
one (countable additivity).

\subsubsection{Bayesian Model}
Each marginal Bayesian model is defined by
a multinomial distribution $\PI[\XI|\ThetaI]$ on $\mI$ categories
and by its conjugate prior
$\PThetaI[\ThetaI|\YI]$ on $\simpI$. The Dirichlet distribution 
is a natural conjugate prior as in \eqref{eq:expfam:densities},
with parameters $(\lambda,\gammaI):=(\lambda,\alpha G_{\indI})$, where
$\alpha\in\mathbb{R}_+$ controls concentration and
an $G_{\indI}\in\simpI$ is the expected value. Since
$\log Z_{\indI}(\thetaI)=0$ for
the Dirichlet distribution, the value of $\lambda$ does not
affect the model and is henceforth omitted. 
Though
$\alpha$ controls the concentration of the model, it
acts linearly on $\thetaI$, in contrast to the nonlinear
influence of $\lambda$ on other conjugate priors. 
It is easy to show
that the multinomial and Dirichlet families so defined
form a projective family of Bayesian models if and
only if the hyperparameters are chosen consistently as
$\gammaI:=\alpha\cdot\gI G_0$ for a fixed $\alpha\in\mathbb{R}_+$ and
some $G_0\in\tspaceD$. 

\subsubsection{Pullback to $M(V)$}
What remains to be done is to ensure that the Dirichlet process
prior $\PThetaD[\ThetaD|\YD=\yD]$ defines a measure on the
set of probability measures. 
\begin{lemma}[Proof: App.~\ref{sec:proofs:examples}]
  \label{lemma:examples:DP:concentration:prior}
  If $V$ is Polish, the countable generating algebra $\mathcal{Q}\subset\borelV$ can be 
  chosen such that, for any charge $G_0$ on $\mathcal{Q}$,
  \begin{equation}
    \PD^{\theta,\ast}[\pMeas(\mathcal{Q})|\YD=(\alpha,G_0)]=1 
    \quad\Leftrightarrow\quad
    G_0\in\pMeas(\mathcal{Q})\;.
  \end{equation}
\end{lemma}
In other words, the prior concentrates on countably additive set functions if
and only if its hyperparameter is countably additive.
We obtain a corresponding concentration result for the sampling model:
\begin{lemma}[Proof: App.~\ref{sec:proofs:examples}]
  \label{lemma:examples:DP:concentration:likelihood}
  Define a relation $\phi\subset V\times\xspaceD$ by means of 
  \begin{equation}
    v\equiv_{\phi}\xD\qquad\Leftrightarrow\qquad\lim\xD=\lbrace v \rbrace
  \end{equation}
  \begin{enumerate}
  \item $\phi$ is a mapping $V\rightarrow\xspaceD$, and a Borel embedding.
  \item If $\thetaD\in M(V)$ is purely atomic, $\PD^{\ast}[\phi(V)|\ThetaD=\thetaD]=1$.
  \end{enumerate}
\end{lemma}
Lemma \ref{lemma:examples:DP:concentration:likelihood} provides a 
suitable embedding $\phi$ for the pullback of $\PD[\,.\,|\ThetaD]$.
For the prior, let $\inclusionT:\pMeas(V)\rightarrow\tspaceD$
be the mapping which takes a probability measure $\nu$ on $\borelV$
to its restriction on $\mathcal{Q}$. By the Carath\'{e}odory extension
theorem, $\inclusionT$ is injective. Since both $\pMeas(V)$ and 
the Borel subset $\pMeas(\mathcal{Q})\subset\tspaceD$ are standard
Borel spaces, $\inclusionT$ is a Borel embedding (cf Sec.~\ref{sec:stochproc:stochproc}).
We set $\txspace:=\lbrace \lbrace v\rbrace|v\in V\rbrace$, which
we identify with $V$, and choose
$\ttspace=\pMeas(V)$ as parameter
space and $\tyspace:=\mathbb{R}_+\times\pMeas(V)$ 
as hyperparameter space. We do not show here that draws from
the Dirichlet process are almost surely discrete, and instead
refer to \citep{Ghosal:2010}.

Sufficient statistics of the marginal models can be defined as
$\SI:I\rightarrow\simpI$ with $A_i\mapsto\delta_{\lbrace A_i\rbrace}$,
\ie the event $\lbrace \XI=A_i\rbrace$ is mapped to a point mass at the singleton $\lbrace A_i\rbrace\in\borelI$.
Define $\tS:V\rightarrow\pMeas(V)$ by $v\mapsto\delta_{\lbrace v \rbrace}$.
For any $v\in V$ and any $I\in D$, there is a unique $A_i\in I$ with $v\in A_i$.
Hence, 
\begin{equation}
        (\gI\circ\inclusionT)\circ\tS = \SI\circ(\fI\circ\phi) \;,
\end{equation}                  
since $\gI\circ\inclusionT$ maps $\nu\in M(V)$ to its evaluation on
the partition $I$, and $\fI\circ\phi$ maps $v\in V$ to the event
$A_i$ in $I$ containing $v$.
Therefore, $\SI$ is a pullback of the projective limit $\SD=\plim\famD{\SI}$.
By Theorem \ref{theorem:proj:lim:sufficient:field} and Proposition
\ref{theorem:suffstat:pullback}, $\tS$ is a sufficient statistic
for the pullback model. By Theorem \ref{theorem:conjugate:projective:limits}
and Proposition \ref{theorem:conjugate:pullbacks}, the pullback model
is conjugate. In summary:
\begin{corollary}
  The pullback Bayesian model defined by $\tP[\tX|\tTheta]$ and 
  $\tPTheta[\tTheta|\tY]$ is a conjugate Bayesian model with
  hyperparameter space $\mathbb{R}_{>0}\times\pMeas(V)$,
  parameter space $\lbrace \nu\in\pMeas(V)|\nu\text{ discrete}\rbrace$, and sample space $V$.
  The posterior index under $n$ observations
  $\tx^{(1)},\dots,\tx^{(n)}\in V$ is
  \begin{equation}
    \alpha G_0\quad\mapsto\quad \alpha G_0+\sum_{i}\tS(\tx^{(i)}) \;,
  \end{equation}
  and $\tS:v\mapsto\delta_{\lbrace v\rbrace}$ is a sufficient statistics of the model.
\end{corollary}
The measure $\tPTheta[\tTheta|\tY=(\alpha,G_0)]$ is, in the terminology of nonparametric
Bayesian statistics, a Dirichlet process with concentration $\alpha$ and base measure $G_0$.

\def\Part{\Pi}
\def\PartI{\Part_{\indI}}
\def\PartJ{\Part_{\indJ}}
\def\mI{m_{\indI}}
\def\mJ{m_{\indJ}}
\def\bgammaI{\bar{\gamma}_{\indI}}
\def\piD{\pi_{\indD}}
\def\kernel{k}
\def\tauI{\tau_{\indI}}
\def\tauJ{\tau_{\indJ}}
\def\tauD{\tau_{\indD}}
\def\thetaI{\theta_{\indI}}
\def\thetaJ{\theta_{\indJ}}
\def\thetaD{\theta_{\indD}}
\def\kI{k_{\indI}}
\def\kJ{k_{\indJ}}
\def\kD{k_{\indD}}
\def\modelI{\model_{\indI}}
\def\SG#1{\mathbb{S}_{#1}}
\def\SGinf{\mathbb{S}_{\infty}}
\def\VP{\mathfrak{S}}
\def\fnn{f_{n+1,n}}
\def\gammaI{\gamma_{\indI}}
\def\gammaD{\gamma_{\indD}}
\def\chull{\mbox{conv}}
\def\TnI{\Tn_{\indI}}
\def\TnD{\Tn_{\indD}}
\def\kernelI{\kernel_{\indI}}
\def\kernelD{\kernel_{\indD}}
\def\yD{y_{\indD}}
\def\yI{y_{\indI}}

\subsection{A Nonparametric Model on Permutations}
\label{sec:examples:cayley}

In this second example, the observations $\tx$ are elements of the infinite
symmetric group $\SGinf$, and the parameters are sequences $\ttheta\in\mathbb{R}^{\infty}$ 
satisfying a certain convergence condition. The infinite symmetric group
is the set of all permutations of the set $\mathbb{N}$ which change an arbitrary
but finite number of elements. 

Models on such infinite permutations are of
potential interest in two contexts: 
(1) Rank data is modeled by permutations,
and a nonparametric approach to ranking problems motivates models on
infinite permutations. In parametric 
rank data analysis, models for ``partial'' data,
\ie data in which part of each ranking is censored, are used to model
``rank your favorite $r$ items out of a total of $n$ items''
\citep{Fligner:Verducci:1986}. In this
case, $n$ is the order of the underlying symmetric group $\SG{n}$,
and $r$ the number of uncensored positions.
\citet{Meila:Bao:2008} 
observe that positing a given set of $n$ items
to choose from makes most partial ranking tasks artificial.
They suggest a nonparametric model on $\SGinf$ to represent
more realistic tasks (``rank your favorite $r$ movies'',
as opposed to ``rank your favorite $r$ out of these $n$ movies'').\\
(2) The cycles of an infinite permutation induce a partition of $\mathbb{N}$, and
random permutations hence induce random partitions. The most prominent example
of such a model is without doubt the Chinese Restaurant Process, 
proposed by Pitman and Dubins as a distribution on infinite
random permutations with uniform marginals \cite{Pitman:2006}.

To construct a Bayesian model on $\SGinf$ by means of a projective limit,
we draw on a beautiful construction recently proposed in
representation theory by \citet*{Kerov:Olshanski:Vershik:1993,
Kerov:Olshanski:Vershik:2004}. This approach constructs a
compactification $\VP$ of $\SGinf$ as a projective limit of
finite symmetric groups; \citet{Kerov:Olshanski:Vershik:2004}
refer to the elements of $\VP$ as ``virtual permutations''.
We construct a Bayesian model by endowing each of the finite
groups with a parametric model based on the Cayley distance,
due to \citet{Fligner:Verducci:1986}. We then give conditions
under which the projective limit concentrates on the subset $\SGinf$.

\subsubsection{Projective Limits of Symmetric Groups}

The projective limit of \citet{Kerov:Olshanski:Vershik:2004} 
assembles the symmetric groups
$\SG{1},\SG{2},\dots$ sequentially, and we hence choose the index set
$D=\lbrace[n]|n\in\mathbb{N}\rbrace$, ordered by inclusion. To define
a projective system, we need a suitable notion of projection
mappings. Given the choice of $D$, it is sufficient to consider
mappings $\fJI$ for $J=[n+1]$ and $I=[n]$, which we more conveniently
denote $\fnn$. Intuitively, the projection should remove the entry
$n+1$ from permutations in $\SG{n+1}$ -- which raises the question
of how to consistently delete $n+1$ from, say, 
$\tiny\bigl(\begin{matrix} 1 & 2& 3& 4\\ 3 & 1 & 4 & 2\end{matrix}\bigr)$
to obtain a valid element of $\SG{3}$. An appropriate projector can be
defined as follows.
Any permutation $\pi\in\SG{n}$ admits a unique representation of the form
\begin{equation}
  \label{eq:permutations:encoding}
  \pi = \sigma_{k_1}(1)\sigma_{k_2}(2)\cdots\sigma_{k_n}(n) \;,
\end{equation}
where $k_i$ are natural numbers with $k_i\leq i$, and 
$\sigma_i(j)$ denotes the transposition of $i$ and $j$.
Hence, the vector $(k_1,\dots,k_n)$ is an encoding of $\pi$.
Let $\psi_n$ be the corresponding mapping defined by 
$\psi_n:\pi\mapsto(k_1,\dots,k_n)$.
Due to the constraint $k_i\leq i$, which makes the
encoding of $\pi$ unique, the mapping is
a bijection $\SG{n}\rightarrow\prod_{m\leq n}[m]$,
and a homeomorphism of Polish spaces if both $\SG{n}$
and the image space are endowed with the discrete topology.
On the encoding $\psi_n\pi$, we
can easily define a natural projection by deleting the
last element, \ie as $(k_1,\dots,k_n,k_{n+1})\mapsto(k_1,\dots,k_n)$,
which is just the product space projector
$\projector_{\ind{[n+1][n]}}$
The projectors $\fnn$ are then chosen as the induced
mappings on the groups, hence 
$\fnn:=\psi_n^{-1}\circ\projector_{\ind{[n+1][n]}}\circ\psi_{n+1}$.
The following diagram commutes:
\begin{diagram}[LaTeXeqno]
  \label{diagram:proj:symmetric:groups}
  \SG{n+1} & \rTo^{\quad\psi_{n+1}\quad} & \prod_{m\leq n+1}[m]\\
  \dTo^{\fnn} & & \dTo_{\projector_{\ind{[n+1][n]}}}\\
  \SG{n} & \rTo^{\psi_n} & \prod_{m\leq n}[m]
\end{diagram}
The projectors $\fnn$ have a natural group-theoretic interpretation:
They remove the element $n+1$ from the cycle containing it.
Intuitively speaking, application of
$\sigma_{k_1}(1),\dots,\sigma_{k_n}(n)$ from the left consecutively 
constructs the
cycles of $\pi\in\SG{n+1}$, pending insertion of the final element
$n+1$ into its respective cycle.
This last step is ommitted
by deleting $\sigma_{k_{n+1}}(n+1)$.
The definition of $\fnn$ is consistent with the
Chinese Restaurant Process \citep{Pitman:2006}: The image measure of the CRP marginal
distribution on $\SG{n+1}$ under $\fnn$ is the CRP marginal on $\SG{n}$.

The projections $\fnn$ determine
$(\xspaceD,\topD) := \plim\famD{\SG{n},\top_n,\fnn}$. 
In the projective limit topology
$\topD$, $\xspaceD$ is 
totally disconnected and compact. In analogy to the finite groups,
$\xspaceD$ is homeomorphic to the product space
$\prod_{m\in\mathbb{N}}[m]$ under the encoding map
$\psi_{\indD}:(\sigma_{k_1}(1)\sigma_{k_2}(2)\cdots)\mapsto(k_1,k_2,\dots)$.
The infinite symmetric group $\SGinf$ is a dense countable
subset of $\xspaceD$.
Unlike $\SGinf$, the space $\xspaceD$ is not a group,
whereas conversely, $\SGinf$ is not compact.
The projective limit $\xspaceD$ can thus be regarded
as an abstract compactification of $\SGinf$.
Elements $\pi\in\xspaceD$ are representable in the form 
\begin{equation}
  \pi = \sigma_{k_1}(1)\sigma_{k_2}(2)\cdots
\end{equation}
and can be interpreted as operations that iteratively permute
pairs of elements ad infinitum. If and only if $\pi\in\SGinf$,  
this process ``breaks off'' after a finite number $n$ of steps,
and the encoding of $\pi$ is of the form
${\psi(\pi)=(k_1,\dots,k_n,n+1,n+2,\dots)}$.

\subsubsection{Distance-Based Models}

A widely used class of probability distributions on finite symmetric
groups are location-scale models of the form
\begin{equation}
  \label{eq:distance:based:model}
p(\pi|\theta,\pi_0)=\frac{1}{Z(\theta)}e^{-\theta d(\pi,\pi_0)} \;,
\end{equation}
where $d$ is a metric on $\SG{n}$. Such models are commonly
referred to as \emph{distance-based} models in the rank data
literature.
\citet{Fligner:Verducci:1986} considered the intersection of this
class with another type of model: Let $W^{(1)},\dots,W^{(k)}$ be a set
of statistics on $\SG{n}$ such that the random variables
$W^{(1)}(\pi),\dots,W^{(k)}(\pi)$ are independent if $\pi$ is
distributed uniformly.
Define a parametric model on $\SG{n}$ as
\begin{equation}
  \label{eq:F:V:model}
  p(\pi|\theta):=\frac{1}{Z(\theta)}
  \exp\bigl(-\sum_{j=1}^k \theta^{(j)}W^{(j)}(\pi)\bigr)
  \qquad\theta\in\mathbb{R}^k \;.
\end{equation}
The moment-generating function $M(\theta)$ of this model
is the product $M(\theta)=\prod_{j}M^{(j)}(\theta^{(j)})$ over the
moment-generating functions $M^{(j)}$ of the
variables $W^{(j)}(\pi)$. Hence, the partition function $Z(\theta)$
of the model factorizes as $Z(\theta)=\prod_j Z^{(j)}(\theta^{(j)})
=\prod_j M^{(j)}(-\theta^{(j)})$, and the statistics
$W^{(j)}(\pi)$ are independent random variables if 
$\pi$ is distributed uniformly.
\citet{Fligner:Verducci:1986} show that this independence is
preserved if the model $p(\pi|\theta)$ is substituted for the
uniform distribution.
The models \eqref{eq:F:V:model} coincide with distance-based models of the
form \eqref{eq:distance:based:model} whenever the metric 
$d(\pi,\pi_0)$, for $\pi_0$ the neutral permutation and $\pi$
uniform on $\SG{n}$, 
is decomposable as a sum of independent random variables
$d(\pi,\pi_0)=\sum_j W^{(j)}(\pi)$. The decomposable metrics
considered in \citep{Fligner:Verducci:1986} are the Kendall
metric and the Cayley metric. We will consider the Cayley
metric $d_{\ind{C}}(\pi,\pi')$ in the following, defined as the minimal number of
(not necessarily adjacent) transpositions required to transform
$\pi$ into $\pi'$. For the neutral permutation $\pi_0$, this
metric satisfies $d_{\ind{C}}(\pi,\pi_0)=(n-\#\text{cycles}(\pi))$.
Consequently, $d_{\ind{C}}$ can be decomposed into a sum of statistics
which count the positions in $\pi$, but discount one element
of each cycle. We hence choose
\begin{equation}
  W^{(j)}:=1-\mathbb{I}\lbrace{k_j=j}\rbrace = \left\{\begin{matrix}0 & & j\text{
    smallest element on its cycle}\\
  1 & & \text{otherwise}\end{matrix}\right\} \;.
\end{equation}
The definition differs slightly from the one given by
\citet{Fligner:Verducci:1986},
who discount the largest element
on each cycle instead. The smallest element is
a more adequate choice in the context of nonparametric constructions,
as it is well-defined for infinite cycles.

Independence of the variables $W^{(j)}$ is easily 
verified by constructing a uniform random permutation
$\pi$ by means of $n$ iterations of the 
Chinese Restaurant Process: In step $j$, 
element $j$ is inserted into the current permutation
by uniformly sampling $U\in [j]$. If $U=j$, the
element is placed on a new cycle. Otherwise, it is
inserted to the immediate right of element $U$ on
the respective cycle. The variables $W^{(j)}$ are 
hence indeed independent and Bernoulli distributed.
Since only $U=j$ creates a new cycle, the Bernoulli
parameters are $\frac{j-1}{j}$.


\def\pTheta{p^{\theta}}

The natural conjugate prior $\PTheta_n$ of the generalized
Cayley model on $\SG{n}$ is given by the conditional measure
$\PTheta_n[\Theta_n|Y_n]$ with density
\begin{equation}
  \pTheta_n(\theta_n|\lambda,\gamma_n)
  :=
  \frac{1}{K_n(\lambda,\gamma_n)}
  \exp\bigl(\sum_j\gamma^{(j)}\theta^{(j)}
  -\lambda\log Z_n(\theta_n)\bigr) \;.
\end{equation}

By means of Lemma \ref{lemma:cond:proj:criterion:2}, we can show:
\begin{lemma}[Proof: App.~\ref{sec:proofs:examples}]
  \label{lemma:example:cayley:projective:families}
  $\famD{P_n[\pi_n|\Theta_n]}$ and
  $\famD{\PTheta_n[\Theta_n|Y_n]}$ are projective families of
  conditional distributions.
\end{lemma}
As projective limits of the two families, we obtain 
the regular conditional probabilities
$\PD[\XD|\ThetaD]=\plim\famD{\P_n[X_n|\Theta_n]}$ on the projective
limit space $\xspaceD=\VP$ of virtual permutations, and
$\PThetaD[\ThetaD|\YD]=\plim\famD{\PTheta_n[\Theta_n|Y_n]}$
on the parameter space $\tspaceD=\mathbb{R}^{\mathbb{N}}$. 

\subsubsection{Pullback to $\SG{\infty}$}

For nonparametric applications, infinite random permutations
are of particular interest, \ie observations generated by the model
should almost surely take values $\XD\in\SGinf$. The model constructed
above can be guaranteed to concentrate on $\SGinf$ by a suitable
choice of pullbacks.

Because the pullback model is supposed to concentrate on $\SGinf$, the Borel
embedding $\phi$ should be the canonical inclusion 
$\phi:\SGinf\hookrightarrow\VP$. The form of the corresponding embedding $\inclusionT$
on parameter space is less apparent:
We observe that an infinite permutation $\pi$ is in $\SGinf$ if and only if the
infinite sequence $(W^{(j)})_j$ contains only a finite number of ones.
For a given parameter $\theta^{(j)}$, the probability that $W^{(j)}(\pi)=1$ is
\begin{equation}
  \mbox{Pr}\lbrace W^{(j)}(\pi)=1\rbrace
  =
  \frac{(j-1)e^{-\theta^{(j)}}}{1+(j-1)e^{-\theta^{(j)}}}
  \;=:
  q_j(\theta^{(j)}) \;,
\end{equation}
and we define a mapping $q:\mathbb{R}^{\mathbb{N}}\rightarrow (0,1)^{\mathbb{N}}$ by
$q(\theta):=(q_j(\theta_j))_{j\in\mathbb{N}}$.
By the Borel-Cantelli lemma, only a finite number of $W^{(j)}$ take value one
if and only if $q(\theta)\in\ell_1$ (cf.\ 
proof of Lemma \ref{lemma:example:cayley:concentration}).
Let $\inclusion_{\ell_1}$ be the canonical inclusion $\ell_1(0,1)\hookrightarrow\mathbb{R}^{\mathbb{N}}$,
and define $\inclusionT:=\inclusion_{\ell_1}\circ q$. 
The pullback parameter space is then
\begin{equation}
        \ttspace
        :=
        \inclusionT^{-1}(0,1)^{\mathbb{N}}
        =
        \lbrace \theta\in\mathbb{R}^{\mathbb{N}}|\;\|q(\theta)\|_{\ell_1}<\infty\rbrace \;.
\end{equation}

A pullback of the model under $\phi$ and $\inclusionT$ is justified by the
following concentration result:
\begin{lemma}[Proof: App.~\ref{sec:proofs:examples}]
  \label{lemma:example:cayley:concentration}
  With $\phi$ and $\inclusionT$ defined as above:
  \begin{enumerate}
  \item The mappings $\phi$ and $\inclusionT$ are Borel embeddings.
  \item $\PD^{\ast}[\SGinf|\Theta=\theta]=1$ if and only if $q(\theta)\in\ell_1(0,1)$.
  \item $\PD^{\theta,\ast}[\ttspace|\YD=(\lambda,\gammaD)]=1$ 
    if $\gammaD\in\ttspace$.
  \end{enumerate}
\end{lemma}
The entire projective limit Bayesian model can therefore
be pulled back under $\phi$ and $\inclusionT$ 
in the sense of \eqref{diagram:pullback:model}.
The pullback model, given by conditionals $\tP[\tX|\tTheta]$ and
$\tPTheta[\tTheta|\tY]$, is parametrized by hyperparameter sequences
$\gammaD$ satisfying $\|q(\gammaD)\|_{\ell_1}<\infty$. The parameter
variable $\tTheta$ almost surely takes values $\theta$ satisfying
$\|q(\theta)\|_{\ell_1}<\infty$, and the observation
variable satisfies $\tX\in\SGinf$ almost surely. 
The sufficient statistics $\SI=S_n$ of the finite-dimensional
models, matching the exponential family representation
\eqref{eq:expfam:densities}, are simply given by $S^{(j)}(\pi)=-W^{(j)}(\pi)$.
By Theorem \ref{theorem:proj:lim:sufficient:field} and
Proposition \ref{theorem:suffstat:pullback}, a sufficient
statistic $\tS:\SGinf\rightarrow\lbrace 0,1\rbrace^{\mathbb{N}}$
of the pullback model is therefore given by the countable vector
$\tS(\pi)$ with components
$\tS^{(j)}(\pi)=\mathbb{I}\lbrace k_j=j\rbrace-1$.
By Theorems \ref{theorem:conjugate:projective:limits}
and \ref{theorem:conjugate:pullbacks}, the model is conjugate.
In summary:
\begin{corollary}
  The pullbacks $\tP[\tX|\tTheta]$ and $\tPTheta[\tTheta|\tY]$
  define a projective limit Bayesian model with
  hyperparameter space $\mathbb{R}_{>0}\times\ttspace$, parameter
  space $\ttspace$, and observation space $\SGinf$. 
  The sequence $\tS$ with components $\tS^{(j)}(\pi)=\mathbb{I}\lbrace k_j=j\rbrace-1$
  is a sufficient statistic of the model, and posterior updates
  under observations $\pi^n=(\pi^{(1)},\dots,\pi^{(n)})$, with $\pi^{(j)}\in\SGinf$,
  are given by
  \begin{equation}
    \label{eq:caley:pullback:post:index}
    \tTn(\pi^n,(\lambda,\tgamma))=\Bigl( n+\lambda, \tgamma+\sum_{j\in\mathbb{N}}\tS^{(j)}(\pi^{(j)})\Bigr) \;.     
  \end{equation}
\end{corollary}
Intuitively, the parameter $\theta^{(j)}$ describes an element-wise concentration.
If all elements of $\thetaI$ are negative in the finite-dimensional model, the
expected value of $\PI[\XI|\ThetaI=\thetaI]$ is an anti-mode \cite{Fligner:Verducci:1986}.
The larger the value of $\theta^{(j)}$, the higher the cost of deviation from
the neutral permutation at position $j$. If such a deviation is observed in $\pi$,
$W^{(j)}(\pi)=1$, and \eqref{eq:caley:pullback:post:index} describes a decrease
of the expected concentration at $j$ in the posterior.

The definition of the sufficient statistics used here closely follows the 
customary presentation in the rank data literature. Alternatively, the model
could be expressed (with a different partition function) in terms of sufficient
statistics $S^{(j)}(\pi)=\mathbb{I}\lbrace k_j=j\rbrace$, which emphasizes the
close relation of the model to the representation \eqref{eq:permutations:encoding}.
Similarly, it may be useful to reparametrize the model by $\vartheta:=q(\theta)$,
such that concentration on $\SGinf$ occurs for convergent parameter sequences.
In its present form, the parameter sequence has to diverge instead
-- concentration on $\SGinf$ requires
$W^{(j)}=0$ eventually, and since the variables $W^{(j)}$ are independent,
this occurs almost surely only for diverging concentration parameters.

\pdfoutput=1

\section{Discussion}

Our results show that conjugate nonparametric Bayesian models, when represented as projective
limits, reflect much of the structure of their parametric counterparts.
In particular, their sufficient statistics and the
updates of posterior parameters are projective limits,
and hence the precise infinite-dimensional
analogues, of the respective functions associated with the marginals.

\subsection{Implications for Model Construction}

The results suggest a construction approach for conjugate nonparametric
models roughly analogous
to the parametric case: On a given type of data, define
a sufficient statistic measuring those properties of the data
considered important; define the corresponding exponential family
model and its canonical conjugate prior; and extend these to an 
infinite-dimensional model by means of a projective limit and a pullback.
In many cases, such constructions may draw on
existing projective limit constructions from various fields of mathematics,
and on well-studied exponential family models to be used
as marginals. 
The construction in
The main technical hurdles in such a construction are
the definition of a suitable projective system, and the proof of
existence of the pullback. The latter step can usually be expected 
to be the more demanding one, since our representation uses the pullback
as a convenient general way to formalize almost sure
properties of the random paths of a stochastic process. This formalization
is particularly useful for our purposes, as it allows to establish
results on sufficiency and conjugacy \emph{assuming} that the
pullback exists for a suitable subset of parameters. 
Actually verifying its existence for a given model, however,
may involve any of the subtleties of stochastic process theory.
As examples such as the Dirichlet process demonstrate, there is often
a compellingly simple intuition as to how a stochastic process model 
behaves, but establishing the mathematical accuracy of this intuition
can pose technical challenges.


\subsection{Interactions in the Posterior}

We close with a heuristic observation that
may warrant rigorous investigation in the future.
If a few exceptional cases are neglected for the sake of argument,
our results imply roughly speaking that the class of conjugate
nonparametric Bayesian models corresponds to those with conjugate
exponential family marginals, and hence to those admitting a
representation of the form \eqref{eq:exp:fam:posterior:index:pullback}.
The posterior updates of the marginals are described by sufficient
statistics whose image has fixed, finite dimension.
For the nonparametric model, these updates can be interpreted as follows:
Suppose the sufficient statistics $\SI$ are bivariate, as for example
in covariance estimation. Censored observations are obtained for
index sets $I_1,I_2,\dots\in D$. In the posterior, an
observation at $I_j$ can affect all dimensions $J\in D$,
through any sufficient statistic $S_{\indK}$ with $I_j,J\po K$.
However, even if an infinite number of repeated observations is obtained
for one and the same $I_j\in D$, the interactions described by each 
individual $S_{\indK}$ affect only a finite
subset of posterior dimensions. 
There may be interesting connections here to a family of results
known as \emph{Pitman-Koopman theory}:
Under suitable regularity conditions, a parametric model admits 
a sufficient statistic of dimension finitely bounded with respect
to sample size if and only if it is an exponential family model 
\cite[e.g.][]{Hipp:1974}. Similar results have been obtained for
certain types of L\'{e}vy processes \cite{Kuechler:1982:1,Kuechler:1982:2}.
In summary, it may be possible to characterize conjugate nonparametric
Bayesian models as models for
which the complexity of interactions in the posterior
is finitely bounded, in a manner which remains to be made precise.

\appendix
\pdfoutput=1

\section{Projective Limits and Pullbacks}
\label{sec:background}

Both projective limits (inverse limits) and pullbacks
are standard techniques in pure mathematics, and projective limits of
probability measures are widely used in probability theory. Since
neither is a standard topic in statistics, though, 
this appendix provides a brief survey of some relevant definitions
and results.

A comprehensive reference on general projective limits is Bourbaki's
\emph{Elements of mathematics}; see \citet{Bourbaki:1968,Bourbaki:1966}
for projective limits of spaces and functions. 
Key references on projective limits of measures are
\citet{Bourbaki:2004},
\citet{Rao:2005,Rao:1971},
\citet{Mallory:Sion:1971}, \citet{Choksi:1958} and
\citet{Schwartz:1973}.
On pullbacks of measures, cf.\
\citet[][Vol.\ I]{Fremlin:MT}.
Both projective limits and pullbacks
are common topics in category theory 
\citep[e.g.][]{MacLane:1998}.

\subsection{Projective Systems and their Limits}
\label{sec:projlim}

A projective limit assembles a mathematical object from a system of 
simpler objects. The assembled object may be an infinite-dimensional space constructed
from finite-dimensional subspaces, a group constructed from subgroups,
a measure assembled from its marginals, or a function defined by
combining functions on subspaces. How the objects are ``glued together''
is defined by specifying a system of mappings, denoted $\fJI$ in
the following, which connect ``larger'' objects to ``smaller'' ones.
These mappings generalize the notion of a projection in a product space.
The notion of ``larger'' and ``smaller'' is defined in terms of a
partial order on the set $D$ of object indices. To admit a proper
definition of a limit, and hence of an extension to infinity, the
index set needs to be directed.

Let $D$ be a set and $\po$ a partial order relation on $D$. The
set is called {\em directed} if for any two elements $I,J\in D$,
there is a $K\in D$ such that $I\po K$ and $J\po K$. 
Let $\lbrace\xspaceI\rbrace_{I\in D}$ 
be a family of sets indexed by a directed set $D$.
Require that for any pair $I\po J$ in $D$, there is a generalized
projection mapping
$\fJI:\xspaceJ\rightarrow\xspaceI$,
\ie a mapping satisfying \eqref{eq:def:canonical:map}.
Then $\lbrace \xspaceI,\fJI | I\po J\in D\rbrace$, in short
$\famD{\xspaceI,\fJI}$, is called a
{\em projective system}.
Define a space $\xspaceD$ as follows:
Let $\lbrace \xI | I\in D\rbrace$ be a collection consisting of a single point each from
the spaces $\xspaceI$, for which
\begin{equation}
  \label{eq:def:projlim:xE}
  \xI = \fJI\xJ \qquad\text{ whenever } I\po J\;.
\end{equation}
Identify any such collection with a point $\xD$, and let $\xspaceD$
be the set of all such points. Then $\xspaceD$ is called the
{\em projective limit} of the system. 
The functions $\fI:\xspaceD\rightarrow\xspaceI$ defined by
$\xD\mapsto\xI$ are called {\em canonical mappings}.

The projective limit $\xspaceD$ is a subset of the product space
$\prod_{I\in D}\xspaceI$. We write $\projectorI$ for the canonical
projection $\projectorI:\prod_{I\in D}\xspaceI\rightarrow\xspaceI$.
The canonical mappings are just the restrictions
$\fI=\projectorI\vert_{\xspaceD}$ of the projections to the projective
limit space. The product space may be interpreted as the set of
all functions $x$ with domain $D$ that take values
$x(I)\in\xspaceI$. Consequently, the projective limit
space is precisely the subset of those functions which commute
with the mappings $\fJI$, in the sense that
$x(I)=(\fJI\circ x)(J)$ whenever $I\po J$.

If the spaces $\xspaceI$ are endowed with additional
structure, and if the canonical
mappings $\fJI$ are chosen to preserve this structure
under preimages, a corresponding
structure is induced on the projective limit space. Two
examples relevant in the following are 
topological and measurable spaces. Suppose that
each space $\xspaceI$ carries a topology $\topI$
and a $\sigma$-algebra $\borelI$.
The system $\famD{\xspaceI,\topI,\fJI}$
is called a {\em projective system of topological spaces}
if each $\fJI$ is $\topJ$-$\topI$-continuous.
The {\em projective limit topology} $\topD$ is defined as
$\topD:=\top( \fI ; I\in D)$, the coarsest topology which
makes all canonical mappings $\fI$ continuous.
Analogously, $\famD{\xspaceI,\borelI,\fJI}$
is {\em projective system of measurable spaces}
if the $\fJI$ are measurable, and $\borelD:=\sigma( \fI ; I\in D)$
is called the {\em projective limit $\sigma$-algebra}.
If the $\sigma$-algebras are the Borel sets generated by the
topologies $\topI$, then $\borelD=\sigma(\topD)$.
The general theme is that the mappings $\fJI$ are chosen to be
compatible with the structure defined on the spaces $\xspaceI$,
and the projective limit structure is the one generated by the
canonical maps $\fI$.
In a similar manner,
projective limits can be defined for a range of other structures,
such as groups (with homomorphisms $\fJI$), etc. 

Suppose now that two families of spaces
$\famD{\xspaceI}$ and $\famD{\yspaceI}$ are jointly indexed by the 
same directed set $D$, and connected by a family $\famD{w_{\indI}}$ of 
mappings. If the mappings
commute with the projection maps, they define a projective
limit mapping between the respective projective limit spaces.
\begin{lemma}[Projective limits of functions {\citep[][III.7.2]{Bourbaki:1968}}]
  \label{lemma:projlim:mappings}
  Let 
  $\mathcal{D}^x:=\famD{\xspaceI,\fJI}$ 
  and 
  $\mathcal{D}^y:=\famD{\yspaceI,\gJI}$ be two projective
  systems with a common index set $D$. For each $I\in D$, let
  $w_{\indI}:\xspaceI\rightarrow\yspaceI$.
  Require that the mappings satisfy
  \begin{equation}
    \label{eq:def:projlim:mappings}
    \gJI\circ w_{\indJ} = w_{\indI}\circ\fJI \;.
  \end{equation}
  Then there exists a unique mapping
  $w_{\indD}:\xspaceD\rightarrow\yspaceE$ such that
  $\gI\circ w_{\indD} = w_{\indI}\circ\fI$
  for all $I$. In other words, the diagram on the right below
  commutes if and only if the diagram on the left commutes
  for all $I\po J\in D$:
  \begin{diagram}[small]
    \xspaceJ &\rTo^{w_{\indJ}} & \yspaceJ & \qquad\qquad & \xspaceD &\rTo^{w_{\indD}} & \yspaceE\\
    \dTo_{\fJI} & &\dTo_{\gJI} & & \dTo_{\fI} & &\dTo_{\gI}\\
    \xspaceI &\rTo^{w_{\indI}} & \yspaceI & & \xspaceI &\rTo^{w_{\indI}} & \yspaceI
  \end{diagram}
\end{lemma}

A number of useful properties of mappings are preserved under
projective limits \citep{Bourbaki:1968,Bourbaki:1966}.
If each $w_{\indI}$ is injective or bijective, then so is $w_{\indD}$.
Projective systems $\mathcal{D}^x$, $\mathcal{D}^y$ of topological
spaces preserve continuity, \ie $w_{\indD}$ is $\topD^x$-$\topD^y$-continuous
if and only if each $w_{\indI}$ is
continuous. Projective systems of measurable spaces
preserve measurability (Lemma \ref{lemma:projlim:measurable:maps}); 
projective systems of algebraic structures preserve
homomorphy, etc. A notable exception is that $w_{\indD}$ need
not be surjective, even if all $w_{\indI}$ are.

In a similar manner, projective limits can be defined for set
functions, and in particular for probability measures $\PI$.
The domains $\xspaceI$ of the maps $w_{\indI}$ above 
are replaced by the $\sigma$-algebras
$\borelI$, and the ranges $\yspaceI$ by $[0,1]$.
We denote by $\fJI(\PJ)$ the image measure under projection,
\ie $\fJI(\PJ)=\PJ\circ\fJI^{-1}$.
\begin{theorem}[Kolmogorov; Bochner \citep{Bochner:1955}]
  \label{theorem:bochner}
  Let $\famD{\xspaceI, \borelI, \fJI}$ be a
  projective system of Polish measurable spaces with countable index set $D$, 
  and $\famD{\PI}$ a family of probability measures on these spaces.
  If the measures commute with projection, that is if
  $\fJI(\PJ)=\PI$ whenever $I\po J$,
  there exists a uniquely defined probability measure $\PD$ on the
  projective limit space $(\xspaceD,\borelD)$ such that
  $\fI(\PD)=\PI$ for all $I\in D$.
\end{theorem}
The image measure $\fJI(\PJ)$ is referred to as a \emph{marginal} of
$\PJ$, and whenever $\xspaceI\subset\xspaceJ$ is exactly the subspace
marginal of $\PJ$ on $\xspaceI$. 
The theorem generalizes to the case of uncountable index sets, but
then requires additional conditions to ensure $\xspaceD\neq\emptyset$.
The most commonly used condition is Bochner's ``sequential
maximality'' \cite{Bochner:1955}. Kolmogorov originally proved the
theorem for product spaces, for which sequential
maximality is automatically satisfied.

\subsection{Pullbacks of Measures and Functions}
\label{sec:pullbacks}

Projective limit constructions of stochastic processes raise
two problems: One is the effective restriction to countable index sets.
The other is that a construction from finite-dimensional
marginals can only express properties of the constructed
random functions that are verifiable at finite subsets of
points (such as non-negativity), but not
infinitary properties (such as continuity, or countable
additivity of set functions). Both problems can be addressed
simultaneously by means of a \emph{pullback},
defined via the following existence result.
\begin{lemma}[Pullback measure {\citep[][Section 132G]{Fremlin:MT}}]
  \label{lemma:pullback}
  Let $\xspace$ be a set, $(\yspace,\borely,\nu)$ be a measure space
  and $\inclusion:\xspace\rightarrow\yspace$ any function.
  If $\inclusion(\xspace)$ has full outer measure under $\nu$, that is if
  $\nu^{\ast}(\inclusion\xspace)=\nu(\yspace)$, there is a uniquely defined measure
  $\tnu$ on $(\xspace,\inclusion^{-1}\borely)$ such that
  \begin{equation}
    \label{eq:pullback}
    \tnu\circ \inclusion^{-1} = \nu\;.
  \end{equation}
\end{lemma}
The measure $\mu$ defined by \eqref{eq:pullback}
is called the {\em pullback} of $\nu$ under $\inclusion$.
If the pullback exists, $\nu$ can be represented as the image measure
$\nu=\inclusion(\tnu)$.
The outer measure condition $\nu^{\ast}(\inclusion(\xspace))=\nu(\yspace)$
ensures that the definition of $\tnu$ by means of the assignment
$\tnu(\inclusion^{-1}A):=\nu(A)$ is unambiguous: If $A,B\in\borelY$ are two sets,
$\inclusion^{-1}A=\inclusion^{-1}B$ does not imply $A=B$. Hence,
$\nu$ may assign different measures to $A$ and $B$, in
which case it is not possible to assign a consistent value to
$\inclusion^{-1}A=\inclusion^{-1}B$ under the pullback. 
However, this problem does not occur on
the image $\inclusion(\xspace)$, since $\inclusion^{-1}A=\inclusion^{-1}B$ {\em does} imply
$(A\symmdiff B)\cap \inclusion\xspace=\emptyset$.
Thus, if $\nu^{\ast}(\inclusion\xspace)=\nu(\yspace)$, any differences
between $A$ and $B$ are consistently assigned measure zero.

The arguably most important application of pullbacks of measures
is the restriction of a measure to a non-measurable subspace:
  Let $\xspace\subset\yspace$ be an arbitrary subspace, and
  $\nu$ a measure on $\yspace$. If the subspace has full
  outer measure $\nu^{\ast}(\xspace)=\nu(\yspace)$, 
  the measure $\nu$ has a uniquely defined pullback
  $\tnu$ under the canonical
  inclusion map $\xspace\hookrightarrow\yspace$.
  The measure $\nu$ lives on the the measurable
  space $(\xspace,\borely\cap\xspace)$, and assigns
  measure $\tnu(A\cap\xspace)=\nu(A)$ to each intersection
  of a measurable set $A\in\borelY$ with $\xspace$.
  Hence, $\tnu$ can be regarded as the restriction of $\nu$
  to $\xspace$.

As for measures, pullbacks can be defined for functions.
Let $\inclusion_X:\txspace\rightarrow\xspace$ and
$\inclusion_Y:\tyspace\rightarrow\yspace$ be two functions.
A \emph{pullback of a function} $f:\xspace\rightarrow\yspace$ is
any function $\tilde{f}:\txspace\rightarrow\tyspace$ for which
the following diagram commutes:
\begin{diagram}[small]
  \txspace & \rTo^{\tilde{f}} & \tyspace\\
  \dTo^{\inclusion_X} & & \dTo_{\inclusion_Y}\\
  \xspace & \rTo_{f} & \yspace
\end{diagram}
Conversely, if $\tilde{f}$ is given, any function $f$ for
which the diagram commutes is called a \emph{pushforward}
of $\tilde{f}$.

The definitions of pullbacks for measures and functions are compatible,
in the sense that the simultaneous pullback 
of a measure and an integrable
function under the same mapping preserves the integral: Let $\yspace=\tyspace=\mathbb{R}$,
and let $(X,\field,\nu)$ be a measure space such that
$\inclusion_X\txspace$ has full outer measure
$\nu^{\ast}(\inclusion_X\txspace)=\nu(\xspace)$. Let $f$ be 
$\field$-measurable, non-negative and $\nu$-integrable.
Then $\tilde{f}$ is $\inclusion_X^{-1}\field$-measurable and
$\tnu$-integrable.
Since $\nu$ is the image measure of $\tnu$ under $\inclusion_X$,
\begin{equation}
  \label{eq:pullback:preserves:integral}
  \int_{\inclusion_X^{-1}C}\tilde{f}\tnu = \int_C fd(\inclusion_X\tnu)
  =\int_C fd\nu \;.
\end{equation}

\section{Proof of Theorem 3}
\label{sec:proof:conjugacy}

\begin{proof}[Proof of (1)]
\def\wD{w_{\indD}}

Let $\kernelI$ be the probability kernel corresponding to $(\TnI)_n$ in
\eqref{eq:def:posterior:index}. Since the kernels 
$\kernelI:\borelI\times\mathcal{W}_{\indI}\rightarrow[0,1]$
live on different spaces $\mathcal{W}_{\indI}$, they do not themselves form
a projective family. To construct projective kernels $\kernelI'$, let
$\TnD:=\plim\famD{\TnI}$ be the projective limit of the posterior indices for each
$n\in\mathbb{N}$. Let $\mathcal{W}_{\indD}:=\plim\famD{\mathcal{W}_{\indI},\hJI}$.
Denote by $\mathcal{R}\subset\mathcal{W}_{\indD}$ the set of possible values of $(\TnD)_n$,
\begin{equation}
  \mathcal{R}:=\cup_n \TnD(\xspaceD^n,\yspaceD)\;.
\end{equation}
For any $\wD\in\mathcal{R}$ and $\AI\in\borelI$, 
define $\kernelI'(\AI,\wD):=\kernelI(\AI,\hI\wD)$. The
functions so defined form a family of kernels 
$\borelI\times\mathcal{R}\rightarrow[0,1]$. This family is projective:
Let $\xD^n\in\xspaceD^n$ and $\yD\in\yspaceD$. Since the posterior indices
are projective,
\begin{equation}
  \begin{split}
  \kernel_{\indJ}(\gJI^{-1}\AI,\hJ\TnD(\xD^n,\yD))
  \eqae&
  \PThetaJ[\gJI^{-1}\AI|\XJ^n=\fJ^n\xJ^n,\YJ=\hJ\yD]\\
  \eqae&
  \PThetaI[\AI|\XI^n=\fI^n,\YI=\hI\yD]\\
  \eqae&
  \kernelI(\AI,\hI\TnD(\xD^n,\yD)) \;.
  \end{split}
\end{equation}
Hence, $\kernel'_{\indJ}(\gJI^{-1}\AI,\wD)=\kernelI'(\AI,\wD)$ for
any $\wD\in\mathcal{R}$, and $\famD{\kernelI'}$
is a projective family. Let
$\kernelD':=\plim\famD{\kernelI'}$ be the projective
limit, as guaranteed by 
Theorem \ref{theorem:projlim:conditionals}.
Regarded as functions on $\abstspace$, the kernels $\kernelI'$ satisfy
\begin{equation}
  \kernelI'(\AI,\TnD(\XI^n(\omega),\YI(\omega))
  \eqae
  \PThetaI[\AI|\XI^n,\YI](\omega) \;.
\end{equation}
Since $\PThetaD[\,.\,|\XD^n,\YD]=\plim\famD{\PThetaI[\,.\,|\XI^n,\YI]}$,
uniqueness up to equivalence in Theorem \ref{theorem:projlim:conditionals}
implies
\begin{equation}
  \kernelD'(A,\TnD(\xD^n,\yD)) 
  \eqae
  \PThetaD[A|\XD^n=\xD^n,\YD=\yD] \;.
\end{equation}
Therefore, $(\TnD)_n$ is a posterior index of the projective
limit model with corresponding kernel $\kernelD'$.

Since the posterior index $(\TnD)_n$ consists of projective limits of
measurable mappings, conjugacy of all marginals implies
$\TnD(\xspaceD^n\times\yspaceD)\subset\yspaceD$ for any $n\in\mathbb{N}$, and hence
conjugacy of the projective limit model.
\end{proof}

For part (2), the existence of a posterior
index for each marginal is established by means of the following lemma.
\begin{lemma}
  \label{lemma:selectors}
  Let $\inclusionX:\txspace\rightarrow\xspace$ and
  $\inclusionY:\tyspace\rightarrow\yspace$ be continuous
  functions between Polish spaces. Suppose that a measurable
  function $\tf:\txspace\rightarrow\tyspace$ is given.
  If $\inclusionX$ is open or closed, there exists a
  pushforward of $\tf$ which is measurable.
\end{lemma}
The proof of the lemma 
draws on the concept of a \emph{selector} \citep[e.g.][]{Kechris:1995}. 
For a given correspondence (equivalence relation) $R$ on a product set
$A\times B$, a selector is function $\beta:A\rightarrow B$
with $f(a)\in R(a)$, \ie an assignment which transforms the
set-valued map $a\mapsto R(a)$ into a function by selecting
a single element of the set $R(a)$ for each $a$.
In our case, the correspondence of interest
is the preimage $\inclusionX^{-1}$.
A selector can be constructed for any correspondence by
invoking the axiom of choice, but will in general be too complicated
to be of any use. Under additional regularity conditions
on the correspondence and the underlying spaces, the 
selection theorem of Kuratowski and Ryll-Nardzewski 
\citep{Kechris:1995}
guarantees the existence
of a Borel-measurable selector. 
\begin{proof}[Proof of Lemma \ref{lemma:selectors}]
  By the selection theorem
  \citep[][Theorem 12.16]{Kechris:1995}, a correspondence 
  between Polish spaces admits a measurable
  selector if it is weakly measurable and its values are closed non-empty
  sets. We have to show that $\inclusionX^{-1}$ satisfies 
  these conditions. The upper inverse under the
  correspondence $\inclusionX^{-1}$ of a set $A\subset\txspace$ is by definition
  $\lbrace x\in\xspace|\inclusionX^{-1}x\subset A\rbrace$, which
  in this case is just $\inclusionX(A)$. If $\inclusionX$ is open,
  the upper inverse $\inclusionX(A)$ of any open set $A\subset\txspace$
  is in $\borel(\xspace)$, which makes
  the correspondence weakly measurable. Similarly, if $\inclusionX$ is closed,
  the upper inverse of any closed set is in $\borel(\xspace)$, hence
  $\inclusionX^{-1}$ is measurable, and in particular weakly
  measurable since $\txspace$ is Polish. The singletons are closed,
  hence by continuity, $\inclusionX^{-1}x$ is closed, and as a
  preimage non-empty.
  We note that the analogous result for pullbacks instead of pushforwards
  follows \emph{mutatis mutandis}.
\end{proof}

\begin{proof}[Proof of (2)]
  Let $\kD$ be the kernel corresponding to the posterior index
  $(\TnD)_n$ which makes the projective limit model conjugate.
  The marginals form a projective system. Hence for any $I\in D$
  and $\AI\in\borel(\tspaceI)$,
  \begin{equation}
    \kD(\gI^{-1}\AI,\hI\TnD(\xD^n,\yD))
    \eqae
    \PThetaI[\AI|\YI=\hI\TnD(\xD^n,\yD)]
  \end{equation}
  is a valid version of the posterior 
  $\PThetaI[\AI|\XI^n=\fI^n\xD^n,\YI=\hI\yD]$. Since the mappings
  are surjective, any hyperparameter $\yI$ and sample $\xI^n$ is
  representable in this form, and the marginal model is closed
  under sampling since $\hI\TnD(\xD^n,\yD)\in\yspaceI$. 
  By the same identity, any measurable mappings $(\TnI)_n$ satisfying
  \eqref{eq:theorem:conjugate:marginals:posterior:index}
  form a posterior index of the marginal model. 
  A mapping $\TnI$ satisfies
  \eqref{eq:theorem:conjugate:marginals:posterior:index}
  if it is a pushforward of $\TnD$. If the canonical mappings
  are open or closed, the existence of such measurable mappings
  $\TnI$ follows from Lemma \ref{lemma:selectors}.
\end{proof}

\section{Proofs for Section 7}
\label{sec:proofs:examples}

\begin{proof}[Proof of Lemma \ref{lemma:examples:DP:concentration:prior}]
  Since $V$ is Polish, its topology is metrizable by some metric $d$.
  The space is separable, and hence has a dense, countable subset
  $U\subset V$. Let $\mathcal{U}$
  be the set of closed $d$-balls of the form
  \begin{equation}
  \mathcal{U} := \lbrace \bar{B}(v,r) | v\in U,
  r\in\mathbb{Q}_+\rbrace \;,
  \end{equation}
  and let $\mathcal{Q}=\mathcal{Q}(\mathcal{U})$ 
  be the smallest algebra containing $\mathcal{U}$.
  Since $\mathcal{U}$ is countable, so is 
  $\mathcal{Q}(\mathcal{U})$.
  The statement of the Lemma then follows from \citep[][Theorem 1]{Orbanz:2010},
  which states that
  $\PD^{\theta,\ast}[\pMeas(\mathcal{Q}(\mathcal{U}))|\YD=(\alpha,G_0)]=1$
  holds if and only if $G_0$ is countably additive on $\mathcal{Q}(\mathcal{U})$.
\end{proof}

\begin{proof}[Proof of Lemma \ref{lemma:examples:DP:concentration:likelihood}]
  (1) 
  \emph{$\phi$ is a mapping:} 
  We have to argue that, whenever the set $W\subset V$ is a singleton, there is
  exactly one $\xD\in\xspaceD$ with $\lim\xD=W$. 
  For any $\xD=\lbrace C_{\indI}|I\in D\rbrace\in\xspaceD$, by definition,
  $\lim\xD\subset C_{\indI}$ for all $I$. Hence, every partition
  $I\in D$ contains exactly one set $\AI$ with $v\in\AI$.
  (Note that no such set need exist if $W$ is not a singleton.)
  Therefore, $\xD:=\lbrace \AI | I\in D\rbrace$ is the only element of
  $\xspaceD$ satisfying $\lim\xD=\lbrace v\rbrace=W$.\\
  \emph{$\phi$ is measurable:}
  Since the $\sigma$-algebra on $\xspaceD$ is the projective limit $\borelD$,
  $\phi$ is measurable if and only if each of the mappings $\fI\circ\phi$
  is measurable. For any $v\in V$, the image $(\fI\circ\phi)(v)=\AI$
  is the unique set $\AI\in I$ for which $v\in\AI$. The preimage of $\AI\in I$
  is therefore simply
  \begin{equation}
    (\fI\circ\phi)^{-1}\lbrace\AI\rbrace =\lbrace v\in V|v\in\AI\rbrace = \AI \;,
  \end{equation}
  and measurable since $\AI\in\mathcal{Q}\subset\borelV$.\\
  \emph{As a mapping onto its image, $\phi$ has a measurable inverse:}
  By definition, $\phi$ is trivially injective. For measurability of the
  inverse on $\phi(V)$, we have to show $\phi(A)\in\borelD\cap\phi(V)$ for
  every $A\in\borelV$, or equivalently, for every $A\in\mathcal{Q}$.
  For any $A\in\mathcal{Q}$, there is some $I\in D$ with $A\in I$, and
  hence $\lbrace A\rbrace\in\borelI$. The singleton $\lbrace A\rbrace$
  is the base of the cylinder 
  $\fJI^{-1}\lbrace A\rbrace =\lbrace \xD |\lim\xD\subset A\rbrace\in\borelD$.
  Then $\phi(A)=\fJI^{-1}\lbrace A\rbrace\cap\phi(V)$ and hence
  $\phi(A)\in\borelD\cap\phi(V)$.
  \\

  {\noindent (2)}
  Let $\thetaD$ be purely atomic of the form 
  $\thetaD=\sum_{i\in\mathbb{N}}c_i\delta_{v_i}$.
  We will show 
  \begin{equation}
    \label{eq:lemma:examples:DP:concentration:likelihood:2:toshow}
    \PD(\lbrace\xD\rbrace|\ThetaD=\thetaD)=c_i
    \qquad\Leftrightarrow\qquad
    \lim\xD=\lbrace v_i\rbrace \;.
  \end{equation}
  The right-hand side is in turn equivalent to $\xD=\phi(v_i)$.
  Given that \eqref{eq:lemma:examples:DP:concentration:likelihood:2:toshow} holds,
  the proof is complete:
  Since $\lbrace v_i \rbrace\subset V$ for all $i\in\mathbb{N}$,
  \eqref{eq:lemma:examples:DP:concentration:likelihood:2:toshow} implies
  \begin{equation}
    1=\sum_i\PD[\lbrace\phi(v_i)|\ThetaD=\thetaD]
    \leq \PD^{\ast}[V|\ThetaD=\thetaD] \leq 1\;.
  \end{equation}
  To verify \eqref{eq:lemma:examples:DP:concentration:likelihood:2:toshow},
  first observe that for any $v\in V$, there is a decreasing
  sequence of sets $Q_n\in\mathcal{Q}$ with $\lim Q_n=\lbrace v \rbrace$.
  To see this, recall the definition of $\mathcal{Q}$: 
  The algebra is generated by compact balls centered at the points in the
  subset $U\subset V$. Since $U$ is dense,
  there is a sequence $u_n\in U$ with $\lim u_n=v$ and
  $d(u_n,v)<\frac{1}{2n}$. Set $Q_n:=\bar{B}(u_n,\frac{1}{2n})$.
  Hence, $v\in Q_n$ for all $n$, and $v\in\lim u_n$. On the other hand, 
  $\bar{B}(u_n,\frac{1}{2n})\subset\bar{B}(v,\frac{1}{n})$,
  and as the balls are compact,
  $\bar{B}(v,\frac{1}{n})\searrow\lbrace v\rbrace$.
  
  Given such a sequence $(Q_n)_n$, there is a sequence
  $I_1\po I_2\po\dots$ of partitions in $D$
  such that $Q_n\in I_n$ for all $n$. In the representation $\xD=\lbrace C_{\indI}|I\in D\rbrace$,
  we therefore have $C_{\indI_n}=Q_n$. For $\xD=\phi(v_i)$,
  \begin{equation*}
  \PD[\lbrace\xD\rbrace|\ThetaD=\thetaD]
  =
  \lim_{n\rightarrow\infty}
  P_{\indI_n}[f_{\indI_n}^{-1}Q_n|\Theta_{\indI_n}=g_{\indI_n}\thetaD]
  =
  \lim_{n\rightarrow\infty}
  \thetaD(Q_n)
  =
  c_i\;.
  \end{equation*}
\end{proof}

\begin{proof}[Proof of Lemma \ref{lemma:example:cayley:projective:families}]
To show that both 
$\famD{P_n[\pi_n|\Theta_n]}$ and
$\famD{\PTheta_n[\Theta_n|Y_n]}$ are projective families of
conditional distributions, we appeal to Lemma 
\ref{lemma:cond:proj:criterion:2}.
First consider the models $P_n[\pi_n|\Theta_n]$.
For $\pi_n=\sigma_{k_1}(1)\cdots\sigma_{k_n}(n)$, the preimage
$\fnn^{-1}\pi_n$ consists of the permutations 
$\pi_{n+1}=\sigma_{k_1}(1)\cdots\sigma_{k_n}(n)\sigma_{m}(n+1)$
for $m=1,\dots,n+1$. For the sampling distributions, fix
$\theta_{n+1}\in\tspace_{n+1}$, and let 
$\theta_n=\projector_{n+1,n}\theta_{n+1}$. Then
\begin{equation}
  \label{eq:cayley:measure:projective}
  \begin{split}
    P_{n+1}[\fnn^{-1}\pi_n|\Theta_{n+1}=\theta_{n+1}] =&
    P_n[\pi_n|\Theta_n=\theta_{n}]
    \frac{(\sum_{m=1}^n e^{-\theta^{(n+1)}})+1}{1+n e^{-\theta^{(n+1)}}}
    \\
    =& P_n[\pi_n|\Theta_n=\theta_{n}]
  \end{split}
\end{equation}
Lemma \ref{lemma:cond:proj:criterion:2}
requires a product space structure of the sample space
and is thus not directly applicable on the groups $\SG{n}$.
However, the encodings $\psi_n$ map into a product space,
and we may equivalently consider the image measures
$\psi_n(P_n)$ on $\prod_{m\leq n}[m]$.
By \eqref{eq:cayley:measure:projective}, the image
measures under $\psi_n$ satisfy
\begin{equation}
  \projector_{n+1,n}\circ\psi_{n+1}(P_{n+1}[\,.\,|\Theta_{n+1}=\theta_{n+1}])
  =
  \psi_n(P_{n}[\,.\,|\Theta_{n}=\theta_{n}])
\end{equation}
which establishes \eqref{eq:lemma:criterion:2:condition}.
By Lemma \ref{lemma:cond:proj:criterion:2},
the images form a projective family 
of conditional probabilities under
the projections $\projector_{n+1,n}$, and hence
by \eqref{diagram:proj:symmetric:groups},
so do $P_n[\pi_n|\Theta_n]$ under $\fnn$.

For the priors, which are defined on the product spaces
$\mathbb{R}^{n-1}$, Lemma \ref{lemma:cond:proj:criterion:2}
can be applied directly. 
Since $Z_n=\prod_j Z^{(j)}$, the partition function
$K_n$ factorizes as $K_n(\lambda,\gamma_n)=\prod_j
K^{(j)}(\lambda,\gamma^{(j)})$. The projection
$(\projector_{n+1,n}\PTheta_{n+1})[\Theta_n|Y_n]$ therefore has density
\begin{equation*}
  \int \frac{\pTheta_n(\theta_n|\lambda,\gamma_n)
    e^{\theta^{(n+1)}\gamma^{(n+1)}-\lambda\log Z^{(n+1)}(\theta^{(n+1)})}}
          {K^{(n+1)}(\lambda,\gamma^{(n+1)})}d\theta^{(n+1)}
          =\pTheta_n(\theta_n|\lambda,\gamma_n)\;,
\end{equation*}
which establishes \eqref{eq:lemma:criterion:2:condition}.
Hence, $\PTheta_n[\Theta_n|\lambda,\gamma_n]$ is a projective
family of conditionals by Lemma \ref{lemma:cond:proj:criterion:2}.
\end{proof}

\begin{proof}[Proof of Lemma \ref{lemma:example:cayley:concentration}]
(1)
As a canonical inclusion, $\phi$ is an embedding and hence a Borel embedding.
Regarding $\inclusionT$, first note that the mapping
$q:\mathbb{R}^{\mathbb{N}}\rightarrow (0,1)^{\mathbb{N}}$ is injective and continuous,
hence measurable. Its image $\ell_1(0,1)$ is a subset of the Polish space $(0,1)^{\mathbb{N}}$,
and since convergence of a sequence in $(0,1)^{\mathbb{N}}$ is
a measurable event in the tail $\sigma$-algebra, $\ell_1(0,1)$
is Borel and hence itself Polish. As a mapping onto its image,
$q$ is surjective, and as a measurable bijection between Polish
spaces, it has a measurable inverse. Since $\inclusion_{\ell_1}$
is again a canonical inclusion, the composition $\inclusionT=\inclusion_{\ell_1}\circ q$
is a Borel embedding.\\
(2)
A virtual permutation $\pi$ is an element of $\SGinf$ if and only if 
$\sum_{j}W^{(j)}(\pi)<\infty$. If this is the case, all but a finite
number of entries of $\pi$ form their own cycle, and hence $\pi\in\SGinf$.
If the sum diverges, at least one cyclic set contains an infinite number
of elements. The random variables $W^{(j)}(\pi)$ are independent under
the model. Hence, by the Borel-Cantelli lemma, the sum converges if and
only if the sum of probabilities $\mbox{Pr}\lbrace W^{(j)}(\pi)=1\rbrace$
converges, \ie if $q(\theta)\in\ell_1$, and hence if $\theta\in\ttspace$.\\
(3)
The random variables $\ThetaD^{(j)}$ are independent
given the hyperparameters. By the zero-one law, the
event $\lbrace\ThetaD\in\ttspace\rbrace=\lbrace q(\ThetaD)\in\ell_1\rbrace$,
\ie the event that the random sequence $(\ThetaD^{(j)})_j$ diverges,
has probability either zero
or one.  The variables have expectation
$\mean{\ThetaD^{(j)}}=\gammaD^{(j)}$.  
Each component of $q$ by definition satisfies $q_j(t)\rightarrow 0$
if $t\rightarrow +\infty$. Hence, $\gammaD\in\ttspace=q^{-1}(\ell_1)$
implies $\gammaD^{(j)}\rightarrow\infty$ as $j\rightarrow\infty$.
Thus for any $\epsilon>0$, the expectations satisfy
$\mean{\ThetaD^{(j)}}>\epsilon$ for a cofinite number of indices $j$,
and $\mbox{Pr}\lbrace\ThetaD\in\ttspace\rbrace=1$.
\end{proof}

\bibliography{references.bib}
\bibliographystyle{latexinclude/natbib}

\end{document}